\documentclass{amsart}
\usepackage{amsmath}
\usepackage{amsfonts,amssymb}
\usepackage[all]{xy}
\usepackage{pb-diagram,pb-xy}

\newtheorem{theorem}{Theorem}[section]
\newtheorem{lemma}[theorem]{Lemma}
\newtheorem{corollary}[theorem]{Corollary}
\newtheorem{proposition}[theorem]{Proposition}
\theoremstyle{definition}
\newtheorem{definition}[theorem]{Definition}

\theoremstyle{remark}
\newtheorem{remark}[theorem]{Remark}

\numberwithin{equation}{section}

\newenvironment{bew}[2]{\removelastskip\vspace{6pt}\noindent
 {\it Proof  #1.}~\rm#2}{\par\vspace{6pt}}
\newenvironment{bteoa}[2]{\removelastskip\vspace{8pt}\noindent
 {\textbf{Theorem A. } #1}\it#2}{\par\vspace{10pt}}
\newenvironment{bteob}[2]{\removelastskip\vspace{8pt}\noindent
 {\textbf{Theorem B. } #1}\it#2}{\par\vspace{10pt}}
\newenvironment{bteoc}[2]{\removelastskip\vspace{8pt}\noindent
 {\textbf{Theorem C. }#1}\it#2}{\par\vspace{10pt}}

\newcommand{\N}{\mathbb{N}}

\newcommand{\diff}{\mathrm{Diff}^{1}_{+}([-1,1])}
\newcommand{\D}{\mathcal{D}}
\newcommand{\I}{\mathcal{I}}

\begin{document}

\title{{Renormalization for critical orders close to $2\mathbb{N}$}}
\author{Judith Cruz}
\address{Departamento Acad\'emico de Matem\'aticas y Estad\'istica, DAME-UNSA, Universidad Nacional de San Agust\'in, Apartado 23, Av. Independencia s/n, Arequipa, Per\'u}
\curraddr{}
\email{jcruzto@unsa.edu.pe}
\thanks{Partially supported by CAPES.}

\author{Daniel Smania}
\address{Departamento de Matem\'atica, ICMC-USP, Universidade de S\~{a}o Paulo, Caixa Postal 668, S\~{a}o Carlos-SP, CEP 13560-970, Brazil}
\email{smania@icmc.usp.br}
\thanks{Partially supported by  FAPESP  2008/02841-4, CNPq 310964/2006-7 and 303669/2009-8.}
\urladdr{http://www.icmc.usp.br/$\sim$smania}


\date{\today}

\dedicatory{}
\begin{abstract}
We study the dynamics of the renormalization operator acting on the
space of pairs $(\phi,t),$ where $\phi$ is a diffeomorphism and
$t\in[0,1],$  interpreted as unimodal maps $\phi\circ q_t$, where
$q_t(x)=-2t|x|^{\alpha}+2t-1$.  We prove the so called complex
bounds for sufficiently renormalizable pairs with bounded
combinatorics. This allows us to show that if the critical exponent
$\alpha$ is close to an even number then the renormalization
operator has a unique fixed point. Furthermore this fixed point is
hyperbolic and its codimension one stable manifold contains all
infinitely renormalizable pairs.
\end{abstract}

\maketitle

\section{Introduction}
The theory of renormalization were motivated by
the conjecture of Feigenbaum and P. Coullet-C. Tresser which stated that the period-doubling operator, acting on the
space of unimodal maps, has a unique fixed-point which is hyperbolic
with an one-dimensional unstable direction.

Lyubich~\cite{Lyu99} proved the Feigenbaum-Coullet-Tresser
conjecture for unimodal maps with {\it even} critical order asserting that the period-doubling fixed point is
hyperbolic, with a codimension one stable manifold (indeed Lyubich proved a far more general result). An extension the results of Lyubich's hyperbolicity to the space of $C^r$ unimodal maps with $r$ sufficiently large were given by E. de
Faria, W. de Melo and A. Pinto\cite{E-W-P}.

All theses results  on the uniqueness,
hyperbolicity and universality  of the fixed point of the renormalization operator  are for
 unimodal maps whose   critical exponent  is an positive even integer. When the order is a non-integer positive integer, very few rigorous results are known. Our goal is to obtain some results in this case.

Fix $\alpha > 1$ and  consider the  class of unimodal
maps $f=\phi\circ q_t:[-1,1]\rightarrow [-1,1],$ where $\phi$ is
orientation preserving  diffeomorphism of the interval $[-1,1]$, $\phi(-1)=-1$, $\phi(1)=1$,
and $q_t(x)=- 2t |x|^{\alpha}+2t-1.$ Note that $q_{t}$ preserves the interval $[-1,1]$
when $t\in [0,1].$ Marco Martens~\cite{MM} proved, based on real methods, the existence of
fixed points to the renormalization operator, for every periodic
combinatorial type, acting on the class of unimodal
maps mentioned above.

It is not clear how to see the renormalization
operator, acting on the class of unimodal maps, as an analytic
operator when the critical exponent $\alpha$ is not an even natural
number. In view of this problem we define a new renormalization
operator, denoted by $\mathcal{\widetilde{R}}_{\alpha},$ in a
suitable space of pairs $(\phi,t),$ where $\phi\circ q_t$ is a
unimodal map.

The advantage of dealing with the new renormalization operator is
that it is a compact complex analytic operator when we endow the ambient
space of pairs $(\phi,t)$ with a structure of a complex analytic
space. The complexification of the renormalization operator is done using a result of complex a priori bounds. Then we can
see that the map $\alpha\mapsto \mathcal{\widetilde{R}}_{\alpha}$ is
a real analytic family of operators. This allow us to use
perturbation methods to solved  the conjecture for the
renormalization operator when the critical exponent $\alpha$ is
close enough to an even natural number. So we stablished
\begin{bteoa} Given a periodic combinatorics $\sigma$, if  $\alpha$ is
close enough to $2 \mathbb{N}$ then some iterate of the renormalization operator associate with $\sigma$  acting on the space of pairs $(\phi, t)$, where $\phi$ is a real analytic map and $t\in[0, 1],$ has a hyperbolic fixed point with a codimension one stable manifold.
\end{bteoa}

\begin{bteob}
For $\alpha$ is
close enough to $2 \mathbb{N},$ the fixed point of the renormalization operator associate with $\sigma$ is unique.
\end{bteob}
Also we stablished the
universality for infinitely renormalizable pairs
\begin{bteoc}
The stable manifold of the fixed point contains all the pairs infinitely renormalizable with the combinatorics of the fixed point.
\end{bteoc}

The structure of the paper is as follows. In the Section 2 we first introduce basic
notions on the renormalization of unimodal maps and unimodal pairs. Then in the Section 3 we state our
results on the hyperbolicity of the fixed point when the critical order is close enough to
an even natural number. We present in the Section 4 the real and complex a
priori bounds, the main tool in the proof of our results. In
the Section 5 we introduce the composition operator, denoted by $L$, which relates the new renormalization operator and the usual one.
For even $\alpha,$ we consider the usual renormalization
operator as an operator acting on the space of holomorphic
functions in the Section 6. Also we show the relations between the two
renormalization operators when the critical exponent is an even
natural number. In the last section we proceed to prove the main
theorems.


\section{Preliminaries}\label{'sec1'}

\subsection{Some notations}

Here the positive integers form the set of natural number denoted in
the standard form by $\mathbb{N}.$ Let $I$ be a bounded interval in
the real line. The $a$-stadium set $D_{a}(I)$ is the set of points in the complex plane whose distance to the
interval $I$ is smaller than $a>0.$ For sets $V$ and $W$ contained
in the complex plane we say the subset $V$ is compactly contained in
$W$ denoting by $V\Subset W.$

The Banach space $\mathrm{C}^{k}([-1,1], \mathbb{R} ),$ $k\in \N,$ is
the set of maps $\mathrm{C}^k$ endowed with the sup norm
$$|f|_{\mathrm{C}^k([-1,1])}=\sup_{x\in[-1,1]}\{|f(x)|,|Df(x)|,\dots,|D^k f(x)|\}.$$
We denote as $\diff$ the set of diffeomorphism $\mathrm{C}^{1}$ that
preserve the orientation of the interval $[-1,1].$ This is an open
subset of the Banach space $\mathrm{C}^{1}([-1,1], \mathbb{R} ).$

\subsection{Renormalization of unimodal maps}

In this part we present usual notions of the renormalization
operator, denoted by $\mathcal{R}.$ We follow the  definitions and notations
 as in A. Avila, M. Martens e W. de
Melo~\cite{A-M-W}. Fix $\alpha > 1,$ the so-called  critical order of the unimodal map. The
parametric unimodal family $ q_t: [-1,1]\rightarrow [-1,1],$ with
$t\in [0,1]$ and critical exponent $\alpha,$ is defined
by
$$q_t(x)=- 2t |x|^{\alpha}+2t-1.$$
The parameter $t$ defines the maximum $ q_t(0)=2t-1.$ Let $$
f=\phi\circ q_t:[-1,1]\rightarrow [-1,1]$$ be the unimodal map where
$\phi\in \diff$ and $\alpha > 1$ is its critical exponent.

A
permutation $\sigma\colon \mathcal{J} \rightarrow \mathcal{J}$, where $\mathcal{J}$ is a finite set with  $q$ elements,  endowed with a total order $\prec$, is called a unimodal permutation with period $q$  if  it
satisfies the following condition. Embedding $\mathcal{J}$ in the real line
preserving the order $\prec~,$ then the graph of the permutation
$\sigma$ on $\mathbb{R}^2$ extends,  by the union of the
consecutive points of the graph of $\sigma$ by segments,
to the graph of a unimodal map. Moreover the period of $\sigma$ is $q$.

A collection $\I=\{I_1,I_2,...,I_q\}$ of closed intervals in
$[-1,1]$ is called a cycle for a unimodal map $f$ if it has the
following properties:
\begin{enumerate}
        \item there exists a repelling periodic point $p\in (-1,1)$
        with $I_q=[-|p|,|p|].$
        \item $f:I_i\rightarrow I_{i+1},$ $i=1,2,...,q-1,$ are difeomorphisms.
        \item $f(I_q)\subset I_1$ with $f(p)\in \partial I_1,$ the
        boundary of $I_1.$
        \item the interiors of $I_1,I_2,...,I_q$ are pairwise
        disjoint.
\end{enumerate}

Consider  the  collection $\mathcal{J}_q=\{1,2,...,q\}$   with the order relation $\prec$ defined by
$$j\prec i ,\,\, j\neq i ,\mbox{ iff } \inf I_j<\inf
I_i.$$ Then  the map  $\sigma: \mathcal{J}_q\rightarrow\mathcal{J}_q$
$$\sigma(i)=i+1 \mbox{ mod } q,$$ is a unimodal permutation. We say that $\sigma=\sigma(\I)$ is the combinatorics of the cycle $\I$.

As direct consequence of the definition of cycle
$\I=\{I_1,I_2,...,I_q\}$ we have
\begin{itemize}
        \item $\I$ inheres an order from $[-1,1],$
        \item the map $$\sigma=\sigma(\I):I_i\mapsto I_{i+1 \mbox{ mod } q} $$
        on $\I$ is an unimodal permutation,
        \item the orientation $$o_{\I}:\I\rightarrow\{-1,1\} $$ is defined such that
        $o_{\I}(I_i)=1 $ when $f^i(p)$ is the left extreme of
        $I_i$ and $o_{\I}(I_i)=-1 $ in other case. So we have the cycle $\I$
        is oriented.
\end{itemize}

\begin{definition}
A unimodal map $f=\phi\circ q_t$ is called renormalizable if it has
a cycle. The first return map to $I_q$ will be, after a re-escaling,
a unimodal map. The prime renormalization period of $f$ is the smallest $q >1$
satisfying the above properties. Define the renormalization operator
$\mathcal{R}$ such that for an unimodal renormalizable map
$f=\phi\circ q_t$ we have that $\mathcal{R}f$ is a unimodal map
defined by
$$\mathcal{R}f(z)=\frac{1}{p}f^{q}(pz),$$ $z\in[-1,1].$ The unimodal
map $\mathcal{R}f$ is called the renormalization of $f. $
\end{definition}

\subsection{Renormalization of a pair}

Consider the set $$ \mathcal{U}=\diff\times[0,1],$$ where an element
$(\phi,t)\in \mathcal{U}$ should be interpretated as the unimodal
map $$ f=\phi\circ q_t:[-1,1]\rightarrow [-1,1],$$ with critical
exponent $\alpha > 1.$ The diffeomorphism $\phi$ is called the
diffeomorphic part of the unimodal map $f$. The metric on
$\mathcal{U}$ is the product metric induced by the norm of the sup
on $\diff$ and the interval metric.

Due the  problem of the non  analyticity of the unimodal map at its critical point when $\alpha \not \in 2 \mathbb{N},$ it is convenient to consider unimodal maps as a pair $(\phi,t).$
\begin{definition}
A pair $(\phi,t)\in \mathcal{U}$ is called renormalizable if
$f=\phi\circ q_t$ is renormalizable. The prime renormalization period
of $(\phi,t)$ is the same of $f.$
\end{definition}

Let $\sigma$ be an unimodal permutation and
$$\mathcal{U}_{\sigma}=\{(\phi,t)\in \mathcal{U}\,| \,f=\phi\circ q_t\mbox{ has a cycle } \I \mbox{
com } \sigma(\I)=\sigma\}.$$

Let $I\subset [-1,1]$ be an oriented interval. We consider the zoom
operator
$$Z_I:\diff\rightarrow \diff,$$ which assign to the
diffeomorphism $\phi:I\rightarrow \phi(I),$ the diffeomorphism
$Z_I(\phi):[-1,1]\rightarrow [-1,1]$ defined by:
$$Z_I(\phi)=A_{\phi(I)}\circ \phi \circ A_{I}^{-1},$$ where the transformation
$A_{J}:J\rightarrow [-1,1]$ is the  unique affine, orientation preserving transformation carrying the closed interval $J$ to the
interval $[-1,1].$ The intervals $I$ e $\phi(I)$ have the same
orientation.

For $(\phi,t)\in \mathcal{U}_{\sigma},$ we define the orientation preserving diffeomorphism
$\phi_0:[-1,1]\rightarrow [-1,1]$ as
$$\phi_0=Z_{\phi^{-1}(I_1)}(\phi).$$
Here $I_0=I_q.$ On the other hand for each $I_i\in \I,$ $i\neq 0,$
we define the orientation preserving diffeomorphism
$q_i:[-1,1]\rightarrow [-1,1]$ e $\phi_i:[-1,1]\rightarrow [-1,1]$ by
$$q_i=Z_{I_i}(q_t)$$ and
$$\phi_i=Z_{q_t(I_i)}(\phi),$$ where $q_t(I_i)$ and
$I_{i+1}$ are orientated in the same direction, this is the
orientation $o(I_{i+1})$ defined by the cycle $\I.$ Furthermore, let
$$t_1=\frac{|q_t(I_q)|}{|\phi^{-1}(I_1)|}.$$

Since $f(I_q)=\phi\circ q_t(I_q)\subset I_1,$ in the definition of
the cycle $\I,$ we have that $t_1\in [0,1].$ This is equivalent to
$q_t(0)\in \phi^{-1}(I_1).$

Now we can define the renormalization operator. The
$\sigma$-renormalization operator, denoted by
$$\mathcal{\widetilde{R}}_{\sigma}:U_{\sigma}\rightarrow U,$$ is
defined by the following expression
  \begin{eqnarray}
   \mathcal{\widetilde{R}}_{\sigma}(\phi,t)=((\phi_{q-1}\circ q_{q_-1})\circ...
   \circ(\phi_{2}\circ q_{2})\circ(\phi_{1}\circ q_{1})\circ
   \phi_{0},t_1).\label{'formulanovoR'}
  \end{eqnarray}

For each $\sigma$ there exists
a unique maximal factorization $\sigma=<\sigma_n,...,\sigma_2,\sigma_1>$ such that
$$\mathcal{\widetilde{R}}_{\sigma}=\mathcal{\widetilde{R}}_{\sigma_n}\circ...\circ
\mathcal{\widetilde{R}}_{\sigma_2}\circ
\mathcal{\widetilde{R}}_{\sigma_1}.$$ A unimodal permutation
$\sigma$ is called prime iff $\sigma=<\sigma>.$ Obviously each
permutation in the maximal factorization is prime. So using primes
unimodal permutations we obtain a partition of the set of
renormalizable pairs in $\mathcal{U}.$
\begin{definition}
The renormalization operator denoted by
$$\mathcal{\widetilde{R}}:\{\mbox{renormalizable pairs}\}=\bigcup_{\sigma\mbox{ prime}
}\mathcal{U}_{\sigma}\rightarrow \mathcal{U},$$ is defined by
$\mathcal{\widetilde{R}}\mid \mathcal{U}_{\sigma}=
\mathcal{\widetilde{R}}_{\sigma}.$
\end{definition}

We say that a pair $(\phi,t)\in\mathcal{ U}$ is $N$-times
renormalizable iff $\mathcal{\widetilde{R}}^n (\phi,t)$ is defined
for all $1\leq n\leq N.$ And $(\phi,t)$ is infinitely renormalizable if it is $N$-times renormalizable for all $N\geq 1.$

\begin{definition}
The set of renormalization times $\{q_{n}\}_{n\in \Lambda},$ with $\Lambda \subset \mathbb{N}$ and where $q_{n}<q_{n+1}$, is the set of integers $q$ such that $f$ is renormalizable of period $q.$
\end{definition}
We say that a pair $N$-times
renormalizable $(\phi,t)\in \mathcal{U},$ for $N$ big enough, has
bounded combinatorics by $B>0$ if  $f=\phi\circ q_t$
satisfies  $q_{n+1}/q_{n}\leq B,$ for
all $1\leq n<N.$

\subsection{Iterating pairs}\label{'sec2'}

Closely following the section~2 in~\cite{A-M-W} we
observe that a sequence of pairs in $\mathcal{U},$ produced by
applying any times the renormalization operator
$\mathcal{\widetilde{R}}_{\alpha},$ is such that each pair has as
first component a decomposition of diffeomorphism and the second a
parameter carrying the information of the unimodal part of the
unimodal map.

Fix a $N$-times renormalizable pair $f=(\phi,t)\in\mathcal{U}$
and let $\I ^n=\{I_1^{n},I_2^{n},...,I_{q_{n}}^{n} \}$ be the cycle
corresponding to $n$-th renormalization, $1\leq n\leq N.$ Each cycle
will be partitioned in sets
$$\I ^n= \bigcup _{k\geq 0}^{N}L_{k}^{n}.$$
The level sets $L_{k}^{n},$ $k\geq 0,$ are defined by induction. Let
$$\I ^0=L_0^0=\{[-1,1]\}.$$ If $\I^{n+1}\ni I_i^{n+1}\subset
I_j^{n}\in L_k^n$ and $0\notin I_i^{n+1}$ then $ I_i^{n+1}\in
L_{k+1}^{n+1}.$ If $0 \in I_i^{n+1}$ then $I_i^{n+1}\in
L_{0}^{n+1}.$ Observe that
$$\I ^n= \bigcup _{k\geq 0}^n L_{k}^{n},$$ for $n\leq N.$ First
to $I_1^n\in \I^n,$ define the orientation preserving diffeomorphism $\phi_0^n:[-1,1]\rightarrow [-1,1]$ where
$$\phi_0^n=Z_{\phi^{-1}(I_1^n)}(\phi).$$ Then for each $I_i^n\in \I^n,$ $i\neq
0,$ define the diffeomorphisms, that preserve orientation,
$q_i^n:[-1,1]\rightarrow [-1,1]$ and $\phi_i^n:[-1,1]\rightarrow
[-1,1]$ by
$$q_i^n=Z_{I_i^n}(q_t)$$ and $$\phi_i^n=Z_{q_t(I_i^n)}(\phi),$$
where $q_t(I_i^n)$ and $I_{i+1}^n$ are oriented in the same
direction, this is with the orientation $o(I_{i+1}^n)$ defined by
the cycle $\I^n.$ Furthermore, define
$$t_n=\frac{|q_t(I_0^n)|}{|\phi^{-1}(I_1^n)|}.$$ The above definitions
describe the first component of $\mathcal{\widetilde{R}}^n (\phi,t)$
that consists of the compositions of the diffeomorphisms $q_i^n$
and $\phi_i^n:$
$$\mathcal{\widetilde{R}}^n(\phi,t)=((\phi_{q_n-1}^{n}\circ q_{q_n-1}^{n})\circ...
\circ(\phi_{2}^{n}\circ q_{2}^{n})\circ(\phi_{1}^{n}\circ
q_{1}^{n})\circ \phi_{0}^{n},t_n). $$

\section{Precise statements of the main results}

Let $\mathcal{B}_V$ be the complex Banach space of holomorphic maps
$f$ defined in a neighborhood $V$ with a continuous extension to
$\overline{V},$ endowed with the sup norm. Let $\mathcal{A}_V$ be
the set of holomorphic maps $\varphi:V\rightarrow \mathbb{C}$ with
continuous extension in $\overline{V},$ where $\varphi(-1)=-1$ and
$\varphi(1)=1.$ This is an affine subspace of $\mathcal{B}_V.$

Denote by $\mathcal{T}_V$ the complex Banach space of holomorphic
maps $\omega\in \mathcal{B}_V $ of the form $\omega=\psi(x^{2n})$ in
a neighborhood of $[-1,1],$ with $\omega(-1)=\omega(1)=0,$ endowed
with the sup norm. Let $\mathcal{U}_V$ be  the set of holomorphic maps
$f:V\rightarrow \mathbb{C}$ with continuous extension in
$\overline{V}$ of the form $f=\psi(x^{2n})$ in a neighborhood of
$[-1,1],$ with $f(-1)=f(1)=-1.$ Then $\mathcal{U}_V$ is an affine
space.

Marco Martens~\cite[Theorem~$2.2$]{MM}  showed that the new
renormalization operator $\mathcal{\widetilde{R}}_{\alpha},$ defined
on the space of pairs whose first component are diffeomorphisms
$\mathcal{C}^2$ in $[-1,1]$ with critical exponent $\alpha>1,$ has
fixed points of any combinatorial type.

From now on we fix a prime combinatorics $\sigma$ with period smaller than $B$.  Denote by $\mathcal{H}_{\alpha}(C,\eta,M),$ $\alpha>1,$ the set of the pairs
$(\phi,t)$ satisfying
\begin{enumerate}
  \item $\phi\in \mathcal{A}_{D_{\eta}},$ where $\phi$ is univalent
  on $D_{\eta}$
  \item $\phi$ is real on the real line
  \item $|\phi|_{C^3([-1,1])}\leq C$
  \item $(\phi,t)$ is $M$-times renormalizable with
  combinatorics $\sigma.$
\end{enumerate}
A pair $(\phi,t)$ satisfying the first three above properties for some $\eta, C$  is
called a unimodal pair.

\begin{theorem}[Complex bounds]\label{'restate1'}
For all $\alpha_{0}>1,$ there exists $\varepsilon=\varepsilon(B),$ $\delta_0=\delta_0(B)>0$ and $C_0=C_0(B)>0$ such for all $\alpha \in (\alpha_{0}-\varepsilon,\alpha_{0}+\varepsilon)$ the following holds: for all
$C>0$ and $\eta>0$ there exists $N_0=N_0(C,\eta)$ such that if $(\phi,t)\in
\mathcal{H}_{\alpha}(C,\eta,M)$ with $M>N_0$ then
$\mathcal{\widetilde{R}}^n_{\alpha}(\phi,t)\in
\mathcal{H}_{\alpha}(C_0,\delta_0,M-n)$ for $M>n\geq N_0.$
\end{theorem}

\begin{remark}\label{remext}
Using methods of the proof of the Theorem~\ref{'restate1'} and
in  Sullivan\cite{Su} and Martens\cite{MM} it is possible to show that the first
component of the $C^3$ fixed points $(\phi^\star,t^\star)$ of $\widetilde{\mathcal{R}}_{\alpha}$
are indeed  analytic and univalent in a neighborhood of the interval
$[-1,1].$   By Theorem~\ref{'restate1'} we have complex bounds for the diffeomorphic  part of this fixed point.  Let $\delta<\delta_0/2$. By
Theorem~\ref{'restate1'} $(\phi^\star,t^\star) $ belongs to $\mathcal{H}_\alpha(C_0,2\delta, \infty)$. In the case when $\alpha \in 2\mathbb{N}$, Sullivan\cite{Su} methods implies that there exists such fixed point.  \end{remark}

Fix  $\delta$ such that $2\delta< \delta_0$.  Define $\tilde{N}_1=N_0(C_0,\delta)$, where $C_0$, $N_0$  is as in Theorem~\ref{'restate1'}.
Let  $(\phi,t) \in \mathcal{H}_\alpha(C_0,\delta, N_1+1)$  with
critical exponent $\alpha$ close enough to $\alpha_0$ and
let
$\I^{\widetilde{N}_1}=\{I_1^{\widetilde{N}_1},I_2^{\widetilde{N}_1},...,I_{q_{_{\widetilde{N}_1}}}^{\widetilde{N}_1}\}$
be the cycle corresponding to $\widetilde{N}_1$-th renormalization.
Then the maps
$\phi_0^{\widetilde{N}_1}=Z_{\phi^{-1}(I_1^{\widetilde{N}_1})}(\phi),$
$q_i^{\widetilde{N}_1}=Z_{I_i^{\widetilde{N}_1}}(q_t)$ and
$\phi_i^{\widetilde{N}_1}=Z_{q_t(I_i^{\widetilde{N}_1})}(\phi),$
have univalent extensions in complex domains such that the
composition
$$(\phi_{q_{\widetilde{N}_1}-1}^{\widetilde{N}_1}\circ q_{q_{_{\widetilde{N}_1}-1}}^{\widetilde{N}_1})\circ...
\circ(\phi_{2}^{\widetilde{N}_1}\circ
q_{2}^{\widetilde{N}_1})\circ(\phi_{1}^{\widetilde{N}_1}\circ
q_{1}^{\widetilde{N}_1})\circ \phi_{0}^{\widetilde{N}_1}$$ is
defined in $D_{2\delta}.$ Then it is possible to choose
$\widetilde{\gamma}_1>0$ small enough such that the operator
$\mathcal{\widetilde{R}}^{\widetilde{N}_1}_{\alpha}$ has a extension
to the ball
$$\widetilde{B}((\phi,t),\tilde{\gamma}_1):=\{(\psi,v)\in
\mathcal{A}_{D_{\delta/2}}\times
\mathbb{C},|(\psi,v)-({\phi}^{*},t^{*})|<\widetilde{\gamma}_1\},$$
as a transformation
$\mathcal{\widetilde{\widetilde{R}}}_{\alpha}:\widetilde{B}(({\phi},t),
\widetilde{\gamma}_1)\rightarrow \mathcal{A}_{D_{2\delta}}\times
\mathbb{C},$ defined by the expression
$$\mathcal{\widetilde{\widetilde{R}}}_{\alpha}(\phi,t)=
((\phi_{q_{_{\widetilde{N}_1}}-1}^{\widetilde{N}_1}\circ
q_{q_{_{\widetilde{N}_1}}-1}^{\widetilde{N}_1})\circ...
\circ(\phi_{2}^{\widetilde{N}_1}\circ q_{2}^{
N_2})\circ(\phi_{1}^{\widetilde{N}_1}\circ
q_{1}^{\widetilde{N}_1})\circ
\phi_{0}^{\widetilde{N}_1},t_{_{\widetilde{N}_1}}),
$$ where
$$t_{_{\widetilde{N}_1}}=\frac{|q_t(I_0^{\widetilde{N}_1})|}{|\phi^{-1}(I_1^{\widetilde{N}_1})|}.$$
So $\mathcal{\widetilde{\widetilde{R}}}_{\alpha}$ is defined in an open neighborhood of $\mathcal{H}_\alpha(C_0,\delta, N_1+1)$ in the space $\mathcal{A}_{D_{\delta/2}}\times \mathbb{C}$.  The natural inclusion $j:\mathcal{A}_{D_{2\delta}}\times
\mathbb{C}\rightarrow \mathcal{A}_{D_{\delta/2}}\times \mathbb{C}$
is a linear compact operator between Banach spaces.

\begin{definition}
The complex
renormalization operator, denoted by $\mathcal{\widetilde{R}},$ is defined
by
$$\mathcal{\widetilde{R}}_{\alpha}=j\circ \mathcal{\widetilde{\widetilde{R}}}_{\alpha}.$$
\end{definition}

\noindent Note that $\mathcal{\widetilde{R}}_{\alpha}$ is a compact operator.\\

\begin{theorem}[Hiperbolicity]\label{'fixedhyp'}
 Fix $r \in \mathbb{N}$. There exists $\eta>0$ and a real analytic map $\alpha\rightarrow
(\phi^*_{\alpha},t^*_{\alpha}),$ where $\alpha \in (2r-\eta,2r+
\eta),$ such that $(\phi^*_{\alpha},t^*_{\alpha})$ is a hyperbolic
fixed point to the operator $\widetilde{\mathcal{R}}_{\alpha},$ with
codimension one stable manifold.
\end{theorem}

The uniqueness of the fixed point to the operator
$\widetilde{\mathcal{R}}_{\alpha}$ is showed in the following
result. The proof will be postpone to the last section.

\begin{theorem}[Uniqueness of the fixed point]\label{uniqueness}Fix $r \in \mathbb{N}$.
For each $\alpha$ close to $2r,$ there exists an unique unimodal fixed
point $(\phi_{\alpha}^*,t_{\alpha}^*)$ of
$\mathcal{\widetilde{R}}_{\alpha},$ and it belongs to
$\mathcal{H}_{\alpha}(C_0,\delta_0,\infty).$
\end{theorem}

We define the stable manifold of the fixed point
$(\phi_{\alpha}^*,t_{\alpha}^*),$ denoted by
$W^s=W^s(\phi_{\alpha}^*,t_{\alpha}^*),$ as the set
$$W^s:=\{ (\phi,t):\mathcal{\widetilde{R}}^n(\phi,t)\rightarrow
(\phi_{\alpha}^*,t_{\alpha}^*)\}$$ By
$$\mathcal{\widetilde{R}}_{\alpha}^n(\phi,t)\rightarrow_n
(\phi_{\alpha}^*,t_{\alpha}^*)$$ we say that
$\mathcal{\widetilde{R}}_{\alpha}^n(\phi,t)$ belongs to
$\mathcal{A}_{D_\delta}\times \mathbb{C}$ for $n$ large enough, and
$\mathcal{\widetilde{R}}_{\alpha}^n(\phi,t)$ converges to
$(\phi_{\alpha}^*,t_{\alpha}^*)$ in $\mathcal{A}_{D_\delta}\times
\mathbb{C}$.

Let $V_1$ be a neighborhood of the fixed point
$(\phi_{\alpha}^*,t_{\alpha}^*)$ of the operator
$\mathcal{\widetilde{R}}_{\alpha}.$ Define
$$\mathcal{N}^k_{V_1}(\phi_{\alpha}^*,t_{\alpha}^*):=\{(\phi,t):\mathcal{\widetilde{R}}^j(\phi,t)\in
V_1, 0\leq j\leq k\}\},$$ where $k \in \mathbb{N}\cup \infty.$
Furthermore, we define
$$\mathcal{N}^{\infty}_{V_1}(\phi_{\alpha}^*,t_{\alpha}^*)
:=\bigcap_{k \in
\mathbb{N}}\mathcal{N}^k_{V_1}(\phi_{\alpha}^*,t_{\alpha}^*),$$ this
is, the set of pairs such that all your iterates stay close to
$(\phi_{\alpha}^*,t_{\alpha}^*).$ So we are ready to define the
local stable manifold.
\begin{definition}
For each $V$ neighborhood of $(\phi_{\alpha}^*,t_{\alpha}^*)$ we
define the corresponding local stable manifold to be
$$W^s_{V_1}(\phi_{\alpha}^*,t_{\alpha}^*):=W^s(\phi_{\alpha}^*,t_{\alpha}^*)
\cap\mathcal{N}^k_{V_1}(\phi_{\alpha}^*,t_{\alpha}^*).$$ As
$(\phi_{\alpha}^*,t_{\alpha}^*)$ is a hyperbolic fixed point, we can
choose $V_1$ such that
\begin{equation}\label{stable} W^s_{V_1}(\phi_{\alpha}^*,t_{\alpha}^*)
:=\mathcal{N}^{\infty}_{V_1}(\phi_{\alpha}^*,t_{\alpha}^*).\end{equation}
\end{definition}

Other important result is the
universality for infinitely renormalizable pairs.

\begin{theorem}[Universality]\label{variedade} Fix $r \in \mathbb{N}$.
For $\alpha$ close to $2r$ we have that all unimodal pairs $(\phi,t)$,
infinitely renormalizable with combinatorics $\sigma$ and order
$\alpha$, belongs to the stable manifold of the unique, unimodal
fixed point $(\phi_\alpha^*,t_\alpha^*)$ of
$\tilde{\mathcal{R}}_\alpha$, this is
\begin{eqnarray*}
\bigcup_{C>0}\bigcup_{\eta>0}\mathcal{H}_{\alpha}(C,\eta,\infty)\subseteq
W^s(\phi_{\alpha}^*,t_{\alpha}^*).
\end{eqnarray*}
\end{theorem}

\section{Real and complex a priori bounds}

We will present the main tool for the development of this work,
the called complex bounds: there exists a complex domain $V\supset
[-1,1]$ such that for $n$ big enough the first component of
$\mathcal{\widetilde{R}}^n(\phi,t),$ where $(\phi,t)\in \mathcal{U}$
satisfying appropriated conditions, is well defined and univalent in
$V.$

The complex bounds has a lot applications in the study of the
renormalization operator $\mathcal{R}_{2r},$ $r\in \mathbb{N}.$ One
of the most important applications is the convergence of the
renormalization operator in the set of the infinitely
renormalization maps and the hyperbolicity of this operator in an
appropriate space. Sullivan~\cite{Su} introduced this
property for the infinitely renormalization maps with bounded
combinatorics.

Others related results about infinitely
renormalizable unimodal maps with no bounded combinatorics, were given by
Lyubich~\cite{Lyu97}, Lyubich e Yampolsky~\cite{Lyu-Yamp}, Graczyk e
Swiatek~\cite{J-G}, Levin e van Strien~\cite{Lv-S}. For multimodal
analytic maps, infinitely renormalizable with bounded combinatorics
Smania~\cite{D1} proved ``complex bounds".

We obtain complex bounds for the first component of
the renormalization operator $\mathcal{\widetilde{R}}_{\alpha}$
which is a univalent map. This tool is useful because it
allow us to define the complex renormalization operator
$\widetilde{R}_{\alpha},$ where the critical exponent is $\alpha>1.$

A main  ingredient in the proof of the Complex Bounds's Theorem
is given in the following lemma where we establish real bounds.
We use this to obtain control on the geometry of the cycles of pairs
$N$-times renormalizables, for $N$ enough big, with bounded
combinatorics by a constant $B>0.$

For a proof of the real bounds see~\cite[Theorem~$2.1$, Chapter VI]{W-V}. First fix the
critical exponent $\alpha>1.$

\begin{lemma}[Real bounds]~\cite{W-V}\label{'aprioribound'}
Let $B>0$ be a constant.  Then there exists $0<b<1 $ with the
following property: for all $C>0,$ there exists $N=N(B,C)\geq1$ such
that if $(\phi,t)\in \mathcal{U}$ is $M$-times renormalizable with
bounded combinatorics by  $B$ with $M>N,$ and
$|\phi|_{\mathrm{C}^3([-1,1])}\leq C,$ we have that
\begin{enumerate}
  \item if $I^{n+1}_{i_l}\subset I^{n}_{j},$ $l=1,\cdots,m_n$ are
  the intervals of the $(n+1)$-th renormalization cycle, where $N\leq n<M,$
  contained in the interval $I^{n}_{j}$ of the $n$-th cycle then
\begin{eqnarray}
  b<\frac{|I^{n+1}_{i_l}|}{|I^{n}_{j}|}<1-b,
\end{eqnarray}
where $l=1,\cdots,m_n,$ for all $N\leq n<M.$
  \item if $J$ is a connected component of $I^{n}_{j}\setminus
  \bigcup_{l=1}^{m_n}I^{n+1}_{i_l}$ then
\begin{eqnarray}
  b<\frac{|J|}{|I^{n}_{j}|}<1-b,
\end{eqnarray}
for all $N\leq n<M.$
\end{enumerate}

\end{lemma}

\begin{remark}\label{'boundperturb'}
Let $\alpha>1$ and $\delta>0$ small enough. We can suppose that a
constant $b<1$ of the Lemma~\ref{'aprioribound'} is the same for the
pairs $(\phi,t)\in \mathcal{U},$ $M$-times renormalizables with
bounded combinatorics by $B,$ for $M>N$ big enough, with critical
order $\widetilde{\alpha},$ where
$|\alpha-\widetilde{\alpha}|<\delta.$
\end{remark}

So we can establish the following result.

\begin{theorem}\label{'domaindef'}
Let $B>0$, $\alpha>1,$ and $C>0.$ There exists $\delta>0$ and
$\varepsilon=\varepsilon(\delta)>0,$ such that the following is
satisfied: for each complex domain $V$ containing the interval
$[-1,1]$ there exists $N=N(B,\alpha,C,V)\geq1$ such that
\begin{itemize}
 \item if $(\phi,t)\in \mathcal{U}$ is $M$-times renormalizable with bounded combinatorics by
 $B,$ for $M>N,$ with critical order
 $\widetilde{\alpha},$ $|\alpha-\widetilde{\alpha}|<\delta$
 and $\phi$ univalent map defined on $V,$ and
 \item $|\phi|_{\mathrm{C}^3([-1,1])}\leq C,$
\end{itemize}
Then for all $ N\leq n< M$ the maps
$\phi_0^n=Z_{\phi^{-1}(I_{1}^n)}(\phi),$ $q_j^n=Z_{I_j^n}(q_t)$ and
$\phi_j^n=Z_{\phi^{-1}(I_{j+1}^n)}(\phi),$ where $j=1,\cdots,q_n-1,$
have univalent extensions for maps defined on a
$\varepsilon$-stadium $D_{\varepsilon}.$
\end{theorem}

\begin{proof}
Let $\delta>0$ and $N\geq 1$ be as in the
Observation~\ref{'boundperturb'}. Consider the sectors in the
complex plane denoted by
$$S_{a}^+=\{ z\in \mathbb{C}: |arg (z)|<\frac{\pi}{a}\}$$
and
$$S_{a}^-=\{ z\in \mathbb{C}: |arg
(-z)|<\frac{\pi}{a}\},$$ where $a>0.$ Suppose that the pair
$(\phi,t)\in \mathcal{U}$ satisfies the hypothesis of the theorem.
We fix $\I ^n=\{I_1^{n},I_2^{n},...,I_q^{n} \}$ the corresponding
cycle to the $n$-th renormalization, $N\leq n\leq M.$ We denote by
$x_j^n$ the boundary point of the interval $I_j^n,$ $j\neq 0,$
nearest to the critical point. There is a level of renormalization
between $N$ and $n$ such that the interval containing the critical
point in this level contains the interval $I_j^n.$ So choose the
first $k>0$ such that $I_0^{n-k}\supset I_j^n,$ where $N\leq n-k<n.$
Then $I_j^n$ is not containing in $I_0^{n-k+1}.$  We have two cases:\\
{\it Case I}. First when $k=1$. By Lemma~\ref{'aprioribound'} we
obtain
\begin{eqnarray*}\frac{dist(0,x_j^n)}{|I_j^n|}&\geq&\frac{|I_0^{n}|}{2|I_j^n|}\\&=&\frac{|I_0^{n}|}{2|I_0^{n-1}|}
\cdot\frac{|I_0^{n-1}|}{|I_j^n|}\\&>&\frac{b}{2}\cdot\frac{1}{(1-b)}.
\end{eqnarray*}
{\it Case II}. For $k>1$. We can see that $I_j^n$ is contained in
some interval $I_{j(n-k+1)}^{n-k+1}\subset I_0^{n-k}-I_0^{n-k+1}.$
Actually the interval $I_j^n$ in contained in a nesting sequence of
intervals of deeper levels. So $I_j^n\subset
I_{j(n-1)}^{n-1}\subset\cdots I_{j(n-k+2)}^{n-k+2}\subset
I_{j(n-k+1)}^{n-k+1}$ and by the Lemma~\ref{'aprioribound'} we have
\begin{eqnarray}
|I_j^n|<(1-b)^{k-1}|I_{j(n-k+1)}^{n-k+1}|\leq
\frac{(1-b)^{k-1}}{2}\cdot(|I_0^{n-k}|-|I_0^{n-k+1}|).\label{'eqbound'}
\end{eqnarray}
From Eq.(~\ref{'eqbound'}) we obtain
\begin{eqnarray*}
\frac{dist(0,x_j^n)}{|I_j^n|}&\geq&\frac{|I_0^{n-(k-1)}|}{2|I_j^n|}\\&>&
\frac{1}{(1-b)^{k-1}}\cdot\frac{|I_0^{n-(k-1)}|}{|I_0^{n-k}|-|I_0^{n-k+1}|}\\&>&
\frac{b}{(1-b)^{k-1}}\label{'cota'}
\end{eqnarray*}
In both cases we obtain
$$\frac{dist(0,x_j^n)}{|I_j^n|}> \frac{b}{2(1-b)}.$$
With this estimative we can define the diffeomorphisms $\phi_0^n,$
$q_j^n$ and $\phi_j^n,$ for $j=1,\cdots,q_n-1,$ in a common domain
in the complex plane. In fact, we know the principal branch of the
logarithm function $\log$ is holomorphic on the set
$\mathbb{C}\setminus \{z\in\mathbb{R}:z\leq 0\}.$ Let
$q_t^{+}:S_{\widetilde{\alpha}}^+\rightarrow \mathbb{C}$ and
$q_t^{-}:S_{\widetilde{\alpha}}^-\rightarrow \mathbb{C}$ be the
univalent maps where
$$q_t^{+}(z)= -2te^{\widetilde{\alpha} \log{z}}+2t-1$$ and
$$q_t^{-}(z)= -2te^{\widetilde{\alpha }\log{(-z)}}+2t-1.$$ We follow the proof
defining a common domain to the maps $q_j^n,$ for
$j=1,\cdots,q_n-1,$ taking in mind two different domains for the
critical exponent $\alpha>1.$ Firstly when $\alpha \in (1,2).$ For
$I_j^n\subset S_{\widetilde{\alpha}}^+ $ (or $I_j^n\subset
S_{\widetilde{\alpha}}^-$) applying the zoom operator $Z_{I_j^n}$
for the diffeomorphisms $q_t^+|_{I_j^n}$ (or $q_t^-|_{I_j^n}$ ). So
we can define $q_j^n=Z_{I_j^n}(q_t^+)$ (or $q_j^n=Z_{I_j^n}(q_t^-)$)
on the set $\epsilon_1$-stadium
$$D_{\epsilon_1}=\{z\in\mathbb{C}:dist(z,[-1,1])<\epsilon_1\},$$
where
$$\epsilon_1= \frac{b}{2(1-b)}.$$
This set contains the interval $[-1,1].$
\\Now for $\alpha\geq 2.$ We consider the distance $a_j$ of the
boundary point $x_j^n$ of the interval $I_j^n\subset
S_{\widetilde{\alpha}}^+$ (or $I_j^n\subset
S_{\widetilde{\alpha}}^-$) to the boundary of the sector
$S_{\widetilde{\alpha}}^+$ (or $S_{\widetilde{\alpha}}^-$). A
simples geometric calculus leads to the following relation
$$a_j=\sin(\frac{\pi}{\widetilde{\alpha}})\cdot dist(0,x_j^n)>
\sin (\frac{\pi}{\alpha-\delta})\cdot\frac{b}{2(1-b)}|I_j^n|,$$
doing a zoom of the diffeomorphisms $q_t^+|_{I_j^n}$ (or
$q_t^-|_{I_j^n}$) and taking
$$\epsilon_1=\sin(\frac{\pi}{\alpha-\delta})\cdot\frac{b}{2(1-b)},$$
so we can define $q_j^n=Z_{I_j^n}(q_t^+)$ (or
$q_j^n=Z_{I_j^n}(q_t^-)$)
on a set $\epsilon_1$-stadium $D_{\epsilon_1}.$ \\
Denote by $\epsilon_j^n$ the distance of the interval
$\phi^{-1}(I_{j+1}^n)$ to the boundary of $V.$ It is clare that
$$\epsilon_j^n\geq dist([-1,1],\partial V)$$ for all
$j=0,\cdots,q_n-1.$ There exists $N_0\geq N$ and $K>1000$ big enough
such that for all $N_0\leq n<M $ we have
$$dist([-1,1],\partial V)\geq K|\phi^{-1}(I_{j+1}^n)|,$$ where $j=0,\cdots,q_n-1.$
Then $\phi_j^n=Z_{\phi^{-1}(I_{j+1}^n)}(\phi),$ for all
$j=0,\cdots,q_n-1$ and for all $N_0\leq n<M,$ are defined on
$$D_{K}=\{z\in\mathbb{C}:dist(z,[-1,1])<K\}.$$
Choose $\varepsilon< \epsilon_1.$ We see that $\phi_0^n,$ and the
maps $q_j^n,$ $\phi_j^n,$ for all $j=1,\cdots,q_n-1,$ are defined on
the $\varepsilon$-stadium
$$D_{\varepsilon}=\{z\in\mathbb{C}:dist(z,[-1,1])<\varepsilon \},$$
for all $N_0\leq n<M.$
\end{proof}

\begin{proposition}\label{'rembound'}
Let $(\phi,t)\in \mathcal{U}$ be as in the statement of the
Theorem~\ref{'domaindef'}. There exists $N_0>0,$ $L>0,$ $H>0$ and
$b<1$ such that the maps $\phi_0^n=Z_{\phi^{-1}(I_{1}^n)}(\phi),$
$q_j^n=Z_{I_j^n}(q_t)$ and
$\phi_j^n=Z_{\phi^{-1}(I_{j+1}^n)}(\phi),$ where $j=1,\dots,q_n-1,$
with univalent extensions to maps defined on a domain
$D_{\varepsilon}$ containing the interval $[-1,1],$ satisfying
\begin{eqnarray}
\sum_{j=0}^{q_n-1}|\phi_j^{n}-\mathrm{id}|_{D_{\varepsilon}}\leq L
b^n,\label{'cal1'}
\end{eqnarray}
and
\begin{eqnarray}
\sum_{I^n_j \in L^n_{k}}|q_i^{n}-\mathrm{id}|_{D_{\varepsilon}}\leq
H(1-b)^{k-1},\label{'cal2'}
\end{eqnarray}
for all $N_0\leq n<M.$
\end{proposition}
\begin{proof}
Let $R=dist([-1,1],\partial V).$ For the first estimative
define
$$K=K(j,n)=\frac{R}{|\phi^{-1}(I_{j+1}^{n})|}.$$ There is
$N_0>1$ such that for all $N_0\leq n<M$ we have
$0<1+\varepsilon<K/4$ and
$$\phi_j^n=Z_{\phi^{-1}(I_{j+1}^n)}(\phi),$$ where
$j=0,\cdots,q_n-1,$ are defined on a ball $B(0,K/2).$ From the
Theorem~\ref{'aproxide'}
\begin{eqnarray}
|\phi_j^n-\mathrm{id}|_{D_{\varepsilon}}\leq
O(\frac{1+\varepsilon}{K}).\label{'aprox1'}
\end{eqnarray}
On the other hand by the real bounds there exists constants
$L_1>0$ and $b<1$ such that
$$\sum_{j=1}^{q_n-1}|I^{n}_{j}|\leq L_1b^n,$$
for all $N_0\leq n<M.$ As the diffeomorphism
$\phi$ has bounded derivative those constants can be adjusted such
that
$$\sum_{j=1}^{q_n-1}|\phi^{-1}(I^{n}_{j})|\leq L_1b^n,$$
for all $N_0\leq n<M.$ So we have
\begin{eqnarray*}
\sum_{j=0}^{q_n-1}\frac{1+\varepsilon}{K}&=&
(1+\varepsilon)\sum_{j=0}^{q_n-1}\frac{|\phi^{-1}(I^{n}_{j+1})|}{R}\\
&\leq&\frac{(1+\varepsilon)}{R}Lb^n,
\end{eqnarray*}
this implies the Eq.~(\ref{'cal1'}).\\
To obtain Eq.~(\ref{'cal2'}), we need analyze two cases: \\
{\bf {\it  Case I}}. For $k>1$ in the Case II of the
Theorem~\ref{'domaindef'} we can define for $I_j^n\subset
I_0^{n-k}\setminus I_0^{n-k+1}$
$$
\begin{array}{cc}
    K_1=K_1(j,n)= & \left\{
                     \begin{array}{ll}
                          \displaystyle\frac{|I_0^{n-k+1}|}{2|I_j^n|} &
                          \hbox{se}\quad
                          1<\alpha<2\\ \noalign{\medskip}
                          \sin(\displaystyle\frac{\pi}{\alpha-\delta}).
                          \displaystyle\frac{|I_0^{n-k+1}|}{2|I_j^n|} & \hbox{if}\quad\alpha\geq2
                      \end{array}
                    \right.
\end{array}$$
Observe that the univalent maps
$$q_j^n=Z_{I_j^n}(q_t),$$ where $j=1,\cdots,q_n-1,$ are defined
in the ball $B(0,K_1/2).$ We have $1+\varepsilon<K_1/4$ then by the
Theorem~\ref{'aproxide'}
\begin{eqnarray}
|q_j^n-\mathrm{id}|_{D_{\varepsilon}}\leq
O(\frac{1+\varepsilon}{K_1}).\label{'aprox2'}
\end{eqnarray}
On the other hand we obtain
\begin{eqnarray*}
\sum_{I_j^n\subset I_0^{n-k}\setminus
I_0^{n-k+1}}\frac{1+\varepsilon}{K_1}&\leq &
(1+\varepsilon)\sum_{I_j^n\subset I_0^{n-k}\setminus
I_0^{n-k+1}}\frac{2|I_j^n|}{|I_0^{n-k+1}|}\\ \noalign{\medskip}
&\leq &\frac{2(1+\varepsilon)}{b}.(1-b)^{k-1}.
\end{eqnarray*}
{\bf {\it  Case II}}. For $k=1.$ From Theorem~\ref{'domaindef'}, for
the interval $I_j^n \subset I_0^{n-1},$ where $N\leq n-k<n<M,$ the
univalent maps $q_j^n=Z_{I_j^n}(q_t)$ are defined in the
$\epsilon_1$-stadium $D_{\epsilon_1},$ where
$\epsilon_1>\varepsilon.$ Remember that $q_j^n(-1)=-1$ and
$q_j^n(1)=1.$ So considering the univalent maps $q_j^n$ defined on
$D_{\epsilon_1}\setminus \{-1,1\},$ by the Montel's
Theorem~\cite{J.C1}, form a normal family. Then there exists $C_1>0$
such that for all interval $I_j^n \subset I_0^{n-1},$ where $N\leq
n-k<n<M,$ we have
$$|q_j^n-\mathrm{id}|_{D_{\varepsilon}}< C_1.$$ So from these two cases we
conclude that there exists $H>0$ such that
\begin{eqnarray}
\sum_{I_j^n\in L^n_{k}}|q_i^{n}-\mathrm{id}|_{D_{\varepsilon}}\leq
H(1-b)^{k-1}.
\end{eqnarray}
\end{proof}

Now we are ready to give the proof of the Complex bounds's Theorem.

\begin{bew}{of the Theorem~\ref{'restate1'}} From Theorem~\ref{'domaindef'}
the maps $\phi_0^n,$ $q_i^n$ e $\phi_i^n,$ for all
$i=1,\cdots,q_n-1,$ have univalent extension on a set
$\varepsilon$-stadium $D_{\varepsilon}$ containing the interval
$[-1,1].$ Let $D_{\varepsilon}\supset
D_{\varepsilon/2}\supset[-1,1]$ be subsets strictly nested. Define
for all $j=1,\cdots,q_n-1$ the diffeomorphisms
$$\Phi_j^n=\phi_{q_n-1}^n\circ q_{q_n-1}^n\circ \cdots\circ\phi_{q_n-j}^n $$
and
$$\Psi_j^n=\phi_{q_n-1}^n\circ q_{q_n-1}^n\circ \cdots\circ\phi_{q_n-j}^n\circ q_{q_n-j}.$$
Define $\rho(\phi_{q_n-1}^n):=\varepsilon/3.$ We are going to construct
by induction the domains to the maps $\Phi_j^n$ and $\Psi_j^n,$ for
$j=1,\cdots,q_n-1.$ First let $j=1.$ By Lemma~\ref{'univmap'} there
exists a constant $K>0$ such that:

A. For $q_{q_n-j}^n:D_{\varepsilon}\rightarrow \mathbb{C}$ and
$\Phi_{j}^n:D_{\rho(\phi_{q_n-j}^n)}\rightarrow \mathbb{C}$ there is
a $\rho(q_{q_n-j}^n)$-stadium $D_{\rho(q_{q_n-j}^n)}\subset
D_{\rho_(\Phi_{j}^n)}$ such that
$$q_{q_n-j}^{n}(D_{\rho(q_{q_n-j}^n)})\subset D_{\rho(\Phi_{j}^n)}.$$
Moreover $$\rho(q_{q_n-j}^n)\geq e^{-k|q_{q_n-j}^{n}-\mathrm{id}|_{D_{\varepsilon}}}\rho(\Phi_{j}^n),$$
 where $\rho(q_{q_n-j}^n)$ and $\rho(\Phi_{j}^n)$ are the
 distances between the boundary of $D_{\rho(q_{q_n-j}^n)}$ and
 $D_{\rho(\Phi_{j}^n)}=D_{\rho(\phi_{q_n-j}^n)}$ respectively to the interval
 $[-1,1].$

Notice that $D_{\rho(q_{q_n-j}^n)}$ is the domain of definition of
the map $\Psi_j^n=\Phi_{j}^n\circ q_{q_n-j}^n.$ So again by the
Lemma~\ref{'univmap'} there exists a constant $K>0$ such that:

B. For $\phi_{q_n-j-1}^n:D_{\varepsilon}\rightarrow \mathbb{C}$ and
$\Psi_j^n:D_{\rho(q_{q_n-j}^n)}\rightarrow \mathbb{C}$ there is a
$\rho(\phi_{q_n-j-1}^n)$-stadium $D_{\rho(\phi_{q_n-j-1}^n)}\subset
D_{\rho(\Psi_j^n)}$ such that
$$\phi_{q_n-j-1}(D_{\rho(\phi_{q_n-j-1}^n)})\subset D_{\rho(\Psi_j^n)}.$$
Moreover $\rho(\phi_{q_n-j-1}^n)\geq e^{-k|\phi_{q_n-j-1}^{n}-\mathrm{id}|_{D_{\varepsilon}}}\rho(\Psi_j^n),$ where $\rho(\phi_{q_n-j-1}^n)$ e $\rho(\Psi_j^n)$ are the distances between the boundary of $D_{\rho(\phi_{q_n-j-1}^n)}$
  and $D_{\rho(\Psi_j^n)}=D_{\rho(q_{q_n-j}^n)}$ respectively to the interval $[-1,1].$
Here $D_{\rho(\phi_{q_n-j-1}^n)}$ is the domain of definition of the
map $\Phi_{j+1}^n=\Psi_j^n\circ \phi_{q_n-j-1}^n.$\\Then for each
$j=2,\cdots,q_n-1$ we apply the Lemma~\ref{'univmap'} for a pair of
maps in $A$ and $B.$ Finally we obtain a $\rho(\phi_0^n)$-stadium
$D_{\rho(\phi_0^n)}$ that is the domain of definition of the map
$$\Psi_{q_n}^n\circ \phi_0^n=\phi_{q_n-1}^n\circ q_{q_n-1}^n\circ
\cdots\circ\phi_{1}^n\circ q_{1}^n\circ\phi_0^n$$ where
\begin{eqnarray}
 \rho(\phi_0^n) & \geq & \prod_{i=1}^{q_n-1}e^{-K|q_i^{n}-\mathrm{id}|_{D_{\varepsilon}}}.
 \prod_{i=1}^{q_n-2}e^{-K|\phi_{i}^{n}-\mathrm{id}|_{D_{\varepsilon}}}.\rho(\phi_{q_n-1}^n)\label{'domain'}
\end{eqnarray}
By Proposition~\ref{'rembound'} there exists $L>0,$ $H>0$ and $b<1$
such that
$$\sum_{I_j^n\in
\mathcal{I}^n_{k}}|q_i^{n}-\mathrm{id}|_{D_{\varepsilon}}<H(1-b)^{k-1}$$
and $$\sum_{i}|\phi_i^{n}-\mathrm{id}|_{D_{\varepsilon}}<Lb^n,$$ for
all $N_0<n<M$ and $1\leq k\leq n.$ Follow from~(\ref{'domain'}) that
\begin{eqnarray*}
 \rho(\phi_0^n)
 & \geq& \frac{\varepsilon}{3}\cdot e^{-K\sum_{k=1}^n \sum_{I^n_j \in
L^n_{k}}|q_i^{n}-\mathrm{id}|_{D_{\varepsilon}}}.
 e^{-K\sum_{i}|\phi_i^{n}-\mathrm{id}|_{D_{\varepsilon}}}\\&=&\frac{\varepsilon}{3}\cdot
 e^{-KH\frac{1}{b}-KL},
\end{eqnarray*}
for all $N_0<n<M.$ Taking $\delta_0=\frac{\varepsilon}{3}\cdot
e^{-KH\frac{1}{b}-KL}$ we have that the first component of the
family $\mathcal{\widetilde{R}}^n (\phi,t)$ has a univalent
extension on the domain $D_{\delta_0}$ that not depend of $n.$
\end{bew}

\begin{corollary}\label{'corRcomplex'}
For all $B>0$ there exists $\delta_0(B)>0,$ and
$N=N(B,C,\delta_0)\geq1$ such that if $(\phi,t)\in
\mathcal{H}_{\alpha}(C,\delta_0/2,M)$ Then if $M> n\geq N$ we have
that the first component of $\mathcal{\widetilde{R}}^n (\phi,t),$ is
defined and univalent in a complex domain $D_{2\delta_0}$.
\end{corollary}

\begin{corollary}\label{'Rcomplexconseq'}
For each unimodal pair $(\phi,t)$, infinitely renormalizable with
bounded combinatorics by $B>0,$ there exists $N$ such that the
sequence consisting of the first component of the pairs
$\mathcal{\widetilde{R}}^n (\phi,t)$, with $n \geq N$, is a
pre-compact family in $D_{2\delta_0}.$
\end{corollary}
\begin{proof} The diffeomorphic part of each $\mathcal{\widetilde{R}}^n
(\phi,t)$ is a diffeomorphisms that preserve the interval $[-1,1].$
Actually this analytic diffeomorphism is a decomposition of
diffeomorphism. By Theorem~\ref{'restate1'} the diffeomorphism part
of each $\mathcal{\widetilde{R}}^n (\phi,t),$ where $n\in
\mathbb{N},$ has univalent extension on a fix $\delta_0$-stadium
$D_{\delta_0}$ containing the interval $[-1,1]$ that no depend of
$n.$ Since each of those transformations fix $-1$ and $1,$ follow of
the Montel's Theorem that with the sup norm on the all holomorphic
functions we have that the first component of the pairs
$\mathcal{\widetilde{R}}_n (\phi,t)$ form a pre-compact family in
$D_{2\delta_0}.$
\end{proof}

\begin{corollary}\label{'restatecomplex'}
There exists $\delta_0,$ $C_0$ and $N_0=N_0(C_0,\delta_0)$ such that
\begin{itemize}
  \item We have $\widetilde{\mathcal{R}}^{N_0}_{\alpha}(\mathcal{H}_{\alpha}(C_0,\delta_0,\infty))
  \subset \mathcal{H}_{\alpha}(C_0,\delta_0,\infty).$
  \item For all $C,$ $\eta$ exists $N=N(C,\eta)$ such that for all
  $k\geq 0$ we have $$\widetilde{\mathcal{R}}^{N+kN_0}_{\alpha}
  (\mathcal{H}_{\alpha}(C,\eta,\infty))\subset \mathcal{H}_{\alpha}(C_0,\delta_0,\infty).$$
\end{itemize}
\end{corollary}

\begin{corollary}
A unimodal pair $(\phi_{\alpha}^*,t_{\alpha}^*)$ such that $\widetilde{\mathcal{R}}_{\alpha}(\phi_{\alpha}^*,t_{\alpha}^*)=(\phi_{\alpha}^*,t_{\alpha}^*)$ belongs to
$\mathcal{H}_{\alpha}(C_0,\delta_0,\infty).$
\end{corollary}

\section{Composition transformation}

When the critical order is an even natural number $\alpha=2r$, $r \in \mathbb{N}$, the  relation between the new renormalization operator
$\widetilde{\mathcal{R}}$ and the usual one $\mathcal{R}$  is given by the
\emph{composition transformation}, denoted by $L ,$ that we will define
here. This allow us to transfer some results of the renormalization
$\mathcal{R}$ to the new operator when the critical exponent is an
even number.

Take $\epsilon>0$ small enough and consider the set
$$D_{\epsilon}([0,1])=\{z\in\mathbb{C}: dist(I,z)<\epsilon\}.$$
Now fix the critical exponent $\alpha=2r,$ with $r\in \mathbb{N}.$
Consider a unimodal map $f=\phi\circ q_t:[-1,1]\rightarrow [-1,1],$
with $t\in [0,1],$ where $ q_t: [-1,1]\rightarrow [-1,1],$ is
$$q_t(x)=- 2tx^{\alpha}+2t-1,$$ with critical exponent $\alpha,$ and
$\phi$ is a diffeomorphism $C^{1},$ that preserve the orientation,
of the interval $[-1,1].$ Denote $A_t(x)=- 2tx+2t-1.$ Then we can
write $f= (\phi\circ A_t)(x^{\alpha}).$ We know that a map
$\tau:z\mapsto z^{\alpha}$ is holomorphic in the complex plane
without zero, and a map
$$A_t:z\mapsto - 2tz+2t-1 ,$$ where $t\in D_{\epsilon}([0,1]),$ is
holomorphic in the complex plane. Notice that $q_t=A_t\circ \tau.$

Let $V\subset \mathbb{C}$ be an open connected  set containing the interval
$[-1,1].$ Let  $\widetilde{V}$ be any open connected set containing the interval
$[-1,1]$ and compactly contained in $V$ such that for every $t\in D_{\epsilon}([0,1])$, with small $\epsilon$, we have that $q_t(\widetilde{V})$ is compactly contained in $V$.

\begin{definition} For $\epsilon$ small,
define the complex analytic composition transformation
\begin{eqnarray*}
L:\mathcal{A}_V\times D_{\epsilon}([0,1])&\rightarrow&
\mathcal{U}_{\widetilde{V}}
\end{eqnarray*}
as $L(\phi,t)(z)=\phi\circ q_t(z),$ for $ z$ in
$\overline{\widetilde{V}} .$
\end{definition}

\begin{remark}\label{'relacL'}
It is easy to see that if $(\phi,t)\in \mathcal{U}$ then
$$L\circ \mathcal{\widetilde{R}}(\phi,t)=\mathcal{R}\circ L(\phi,t).$$
\end{remark}

\begin{proposition}[Injectivity of L]\label{'unipointren'}
Let $\alpha > 1$. If $\phi,\widetilde{\phi} \in
\mathcal{A}_V$ and $t,\widetilde{t} \in \mathbb{C}\setminus \{0\} $
are such that
$$\phi\circ q_t=\widetilde{\phi}\circ q_{\widetilde{t}}$$ on $[-1,1],$
then $\phi=\widetilde{\phi}$ and $t=\widetilde{t}.$
\end{proposition}
\begin{proof}
Suppose that $\phi\circ q_t=\widetilde{\phi}\circ
q_{\widetilde{t}},$ on $[-1,1].$ Then for all $y\in [0,1]$
we obtain
$$ A_t(y)=\phi^{-1}\circ \widetilde{\phi}\circ A_{\widetilde{t}}(y).$$
Take $$y=\frac{-x+2\widetilde{t}-1}{2\widetilde{t}},$$ where $x\in
[-1,1].$ It is not difficult to verify
\begin{eqnarray}
\frac{t}{\widetilde{t}}x+\frac{t}{\widetilde{t}}-1=\phi^{-1}\circ
\widetilde{\phi}(x).
\end{eqnarray}
Since $\phi^{-1}\circ \widetilde{\phi}(1)=1$ and
$\phi^{-1}\circ \widetilde{\phi}(-1)=-1$ then $t=\widetilde{t}$ and
$\phi^{-1}\circ \widetilde{\phi}=id.$
\end{proof}

By definition of $L$ for each $(\omega,v)$ in $
T_{\phi}\,\mathcal{A}_V\times \mathbb{C}$ we have
\begin{eqnarray*}
\D L(\phi,t)(\omega,0)&=&\frac{d}{du}[(\phi+u\omega)\circ q_t]|_{u=0}\\
&=&\frac{d}{du}[\phi\circ q_t+u(\omega\circ
q_t)]|_{u=0}\\&=&\omega\circ q_t
\end{eqnarray*}
and by the chain rule
\begin{eqnarray*}
 \D L(\phi,t)(0,v)&=&\frac{d}{du}[\phi\circ q_{t+uv}]|_{u=0}\\
&=&D\phi(q_t)\cdot\frac{d}{du}(q_{t+uv})|_{u=0}
\end{eqnarray*}
where $$\frac{d}{du}(q_{t+uv}(x))|_{u=0}=2v(1-x^{2r}).$$ So the
derivative of $L$ in $(\phi,t)\in \mathcal{A}_V\times
D_{\epsilon}([0,1])$ is the analytic operator
$$\D L(\phi,t):\mathcal{F}_{V}\times \mathbb{C}\rightarrow
\mathcal{T}_{\widetilde{V}}$$ given by
\begin{eqnarray}
\D L(\phi,t) (\omega,v)(x)=\omega\circ q_t(x)
+D\phi(q_t(x)).2v(1-x^{2r}),\label{'ecdis1'}
\end{eqnarray}
for all $x$ in $\widetilde{V}.$

In the following propositions we will prove some properties of the
differential $DL.$

\begin{proposition}\label{'propinj'}
Let $(\phi,t)\in \mathcal{A}_V\times (0,1].$ The operator $\D
L(\phi,t)$ is injective.
\end{proposition}
\begin{proof}
Suppose that $\D L(\phi,t) (\omega,v)(z)=0,$ for all $z\in
\widetilde{V}.$ Then from the Eq.~(\ref{'ecdis1'}) we have
\begin{eqnarray}
\omega\circ q_t(z)&=&-D\phi(q_t(z)).2v(1-z^{2r})\\
\omega\circ
q_t(z)&=&-D\phi(q_t(z)).\frac{v}{t}(q_t(z)+1)\label{'eqinj'}
\end{eqnarray}
From Eq.~(\ref{'eqinj'}) we have that for every $y\in
[-1,q_t(0)]\subset [-1,1]$
\begin{eqnarray}
\omega(y)=-D\phi(y).\frac{v}{t}(y+1),\label{'eqinj2'}
\end{eqnarray}
Since $\omega(y)$ and $D\phi(y).\frac{v}{t}(y+1),$ are analytic on
$[-1,1]$ we have that the Eq.~(\ref{'eqinj2'}) is satisfied for
every $y\in [-1,1].$ If we take $y=1$ we obtain
$$\omega(1)=-D\phi(1)\frac{v}{t}(2)$$ and since that $\omega(1)=0$ and $D\phi(1)\neq 0$
we have $v=0.$ On the other hand the condition $v=0$
in~(\ref{'eqinj2'}) and the analycity of $\omega$ on $V$ imply
$\omega(y)=0,$ for all $y\in V.$ So $\D L(\phi,t)$ is injective.
\end{proof}

\begin{lemma}\label{'lemdenseL'}
Let $(\phi,t)\in \mathcal{A}_V\times (0,1].$ the image of the
operator
$$\D L(\phi,t):\mathcal{F}_V\times \mathbb{C}\rightarrow
\mathcal{T}_{\widetilde{V}}$$ is dense.
\end{lemma}
\begin{proof} It is no difficult to prove that the set of polynomial vector
fields is dense in $\mathcal{T}_{\widetilde{V}}$ (see~\cite{Ju}). So
will be sufficient to show that for all polynomial vector field
$\widetilde{\omega}\in \mathcal{T}_{\widetilde{V}},$ there exists
$(\omega,v)\in \mathcal{F}_V\times \mathbb{C}$ such that
$$\D L(\phi,t) (\omega,v)=\widetilde{\omega}.$$ Since
$\widetilde{\omega}$ is the form $\widetilde{\omega}=\psi(x^{2r}),$
where $\psi$ is a polynomial vector field in a neighborhood of
$[0,1]$, we can write this as $\widetilde{\omega}=\beta\circ q_t,$
where $\beta=\psi\circ A_t^{-1}$ is a polynomial vector field. Take
$$\omega(y)=\beta(y)-D\phi(y)\frac{\beta(1)}{2D\phi(1)}(1+y)$$ and
$$v=\frac{\beta(1)t}{2D\phi(1)}.$$
\end{proof}

Remember that the continuous linear operator $T$ from a  Banach
space $E$ to a  Banach space $F$ is compact if, for each bounded
sequence $\{x_n\}$ in $E,$ the sequence $\{Tx_n\}$ contains a
convergence subsequence in $F.$

\begin{lemma}\label{'lemcompactness'}
Let $(\phi,t)\in \mathcal{A}_V\times D_{\epsilon}([0,1]).$ The
operator
$$\D L(\phi,t):\mathcal{F}_V\times \mathbb{C}\rightarrow
\mathcal{T}_{\widetilde{V}}$$ is compact.
\end{lemma}
\begin{proof} Let $\{(\omega_i,v_i)\}\subset \mathcal{F}_V\times \mathbb{C}
$ be a bounded sequence, this is there is a constant $B>0$ such that
$$|\omega_i|_{V},|v_i| < B,$$
for all $i>0.$ By definition of $L,$ we have that $\{\D
L(\phi,t)(\omega_i,v_i)\}$ is a sequence of analytic vector fields
on $\widetilde{V}.$ We took $\widetilde{V}$ compactly contained in
$V$ such that $q_t(\widetilde{V})$ is compactly contained in $V.$
Now we take a open subset $U\Supset\widetilde{V}$ compactly
contained in $V$ such that $q_t(U)$ is compactly contained in $V.$
Then
$$|\omega_i|_{V}=\sup_{x\in V} |\omega_i(x)|\geq \sup_{x\in q_t(U)} |\omega_i(x)|=
\sup_{y\in U} |\omega_i\circ q_t(x)|= |\omega_i\circ q_t|_{U}.$$ So
$\omega_i\circ q_t$ is bounded on $U$ by $B.$ Since $D\phi$ is
bounded in $q_t(U)\subset V,$ by the Eq.~(\ref{'ecdis1'}) there exists
$C>0$ such that
$$|\D L(\phi,t)(\omega_i,v_i)|_{U}< C$$ for all $i>0.$
Since a uniformily bounded sequence of analytic maps in $U$ is a normal family in $U,$ all
subsequences of $\{\D L(\phi,t)(\omega_i,v_i)\}$ has a convergence
subsequence on $\widetilde{V}$.
\end{proof}

\section{Complex renormalization operators $\mathcal{R}$ and $\tilde{\mathcal{R}}$}

Fixing the critical exponent $\alpha=2r$ where $r\in \mathbb{N},$ we
are going to consider the renormalization operator $\mathcal{R}_{\alpha},$
defined in the Section~$1$, as an operator acting on the space of
holomorphic functions. In the last part of this section, we show
that when $\alpha>1$ is even, the spectrum of $D
\mathcal{R}_{\alpha}$ and $D\mathcal{\widetilde{R}}_{\alpha}$
coincides in the respective fixed points of the renormalization
operators $\mathcal{R}_{\alpha}$ and
$\mathcal{\widetilde{R}}_{\alpha}.$

\subsection{Complex operator $\mathcal{R}_{\alpha}$}\label{'seccomplexR'}

Based in real methods Marco Martens~\cite{MM} proved the existence
of the fixed points of the renormalization operators
$\mathcal{R}_{\alpha},$ of any combinatorial type, acting in the
space of smooth unimodal maps with critical exponent $\alpha>1.$
From definition of the renormalization operator
$\mathcal{R}_{\alpha}$ we have that it has a fixed point, denoted by
 $f^*,$ satisfying
$$f^*(z)=\frac{1}{p}{f^*}^{q}(pz),$$
for some $p\in(-1,1)$ such that ${(f^*)}^{q}(p)=p.$

Given an analytic function $f:V\rightarrow \mathbb{C},$ where $V$ is
an open set, define the open set
$$\mathcal{D}^{n}_{V}(f):=\bigcap^{n-1}_{i=0}f^{-i}V.$$
Given a subset $V\subset\mathbb{C}$ and $\lambda\in \mathbb{C},$
denote by $\lambda V:=\{\lambda x:x\in V\}.$

As a consequence of the complex bounds of Sullivan~\cite{Su}, fixing
the critical exponent $\alpha=2r$ where $r\in \mathbb{N},$ for all
$\varepsilon>0$ small enough there exists $N_1=N_1(\varepsilon)>0$
large enough such that the domains $D_{\varepsilon/2}\Subset
D_{\varepsilon}\Subset D_{2\varepsilon},$ satisfying:
\begin{itemize}
  \item $f^*$ has a continuous extension to $\overline{D}_{\varepsilon}$ which is holomorphic in
  $D_{\varepsilon},$ and has a unique critical point in e $D_{\varepsilon}.$
  \item we have $$p^{N_1}D_{2\varepsilon}\Subset
  \mathcal{D}^{q^{N_1}}_{D_{\varepsilon/2}}(f^*),$$ in other words, we can iterate $f^*:D_{\varepsilon/2}\rightarrow
  \mathbb{C}$ at least $q^{N_1}$ times on a domain $p^{N_1}\overline{D}_{2\varepsilon}.$
\end{itemize}

Now we can define the complex renormalization operator acting on the
holomorphic functions close enough to $f^*.$ Observe that it is
possible to choose $\gamma_1 =\gamma_1(\varepsilon,N_1)>0$ small
enough such that for each $f$ in the ball of center $f^*$ and radius
$\gamma_1$ denoted by
$$B_{\varepsilon/2}(f^*,\gamma_1):=\{f\in
\mathcal{U}_{D_{\varepsilon/2}},|f-f^*|<\gamma_1\},$$ the following
is satisfied:
\begin{itemize}
  \item there exists an analytic continuation $p_f$ of the periodic point $p$
  of $f^*,$ this is $f^q(p_f)=p_f$ and $p_f\sim p.$
  \item we have $p_f D_{2\varepsilon}\Subset
  D^{q^{N_1}}_{\varepsilon/2}(f).$
\end{itemize}
We define the complex analytic operator
$\mathcal{\widehat{R}}_{2r}:B_{\varepsilon/2}(f^*,\gamma_1)\rightarrow
\mathcal{U}_{D_{2\varepsilon}}$ as
\begin{eqnarray}
\mathcal{\widehat{R}}_{2r}(f)(z):=
\frac{1}{p_f}f^{q^{N_1}}(p_fz).\label{'operR'}
\end{eqnarray}
So we define the complex analytic extension of the renormalization
operator $\mathcal{R}_{2r}$ as
$$\mathcal{R}_{2r}:=i\circ \mathcal{\widehat{R}}_{2r},$$ where $i:\mathcal{U}_{D_{2\varepsilon}}\rightarrow\mathcal{U}_{D_{\varepsilon}}$ is the inclusion. Note that $i$ is a compact linear transformation.
\begin{remark}
Notice that we are free to choose $\varepsilon>0,$  $N_1>0,$ and
$\gamma_1>0.$ In the section~\ref{'operRenovo'} we will do a
convenient to choose those constants.
\end{remark}

Edson de Faria, W. de Melo and A. Pinto~\cite{E-W-P}, with the help
of real and complex a priori bounds of Sullivan~\cite{Su} and the
result of hyperbolicity of Lyubich~\cite{Lyu99} (also
see~\cite{D2}), proved the hyperbolicity of the fixed point of the
renormalization operator with respect an iterate of the
renormalization operator acting on the space
$\mathcal{U}_{D_{\varepsilon/2}}$ for some $\varepsilon>0.$ More
precisely the Theorem  2.4 em~\cite{E-W-P} claims:

\begin{theorem}~\cite{E-W-P}\label{'hyperbolicity'}
For $\varepsilon>0$ small enough, there exists a positive number
$N_1=N_1(\varepsilon)>0$ and $\gamma_1>0$ with the following
property. The real analytic compact operator
$\mathcal{R}_{2r}:B_{\varepsilon/2}(f^*,\gamma_1)\rightarrow
\mathcal{U}_{D_{\varepsilon/2}},$ defined by Eq. (\ref{'operR'}),
has a unique hyperbolic fixed point $f^*=\phi^*\circ q_{t^*} \in
B(f^*,\gamma_1)$ with codimension one stable manifold.
\end{theorem}

\subsection{Relating the complex operators $\mathcal{R}_{2r}$
and $\widetilde{\mathcal{R}}_{2r}$}\label{'operRenovo'}

Now consider the critical exponent $\alpha=2r,$ where $r\in
\mathbb{N}.$ With the same notation of the Section~\ref{'seccomplexR'}, choose
$\delta>0,$ $\widetilde{N}_1$. Choose  $\widetilde{\gamma}_1$ such that $\tilde{\mathcal{R}}$ is defined in
$$\widetilde{B}((\phi^\star,t^\star),\tilde{\gamma}_1):=\{(\phi,t))\in
\mathcal{A}_{D_{\delta/2}}\times
\mathbb{C},|(\phi,t)-({\phi}^{*},t^{*})|<\widetilde{\gamma}_1\},$$
where $({\phi}^{*},t^{*})$ is the unique fixed point of $\tilde{\mathcal{R}}_{2r}$. Such fixed point exists due Remark \ref{remext}. The uniqueness follows from the uniqueness of the fixed point of $\mathcal{R}$ and the injectivity of $L$.

Let $\varepsilon_0$ be such that
$$q_{t^*}(\overline{D}_{2\varepsilon_0})\subseteq D_{\delta/2}.$$
And choose $\varepsilon<\varepsilon_0,$ $N_1$ and $\gamma_1>0$ as in
the Theorem~\ref{'hyperbolicity'}. Then we choose
$N=N_1.\widetilde{N_1}$ and consider this iteration.

Since $\mathcal{\widetilde{\widetilde{R}}}_{2r}$ is analytic there
exists $\widetilde{C}>1$ such that
$$\mathcal{\widetilde{\widetilde{R}}}_{2r}(\widetilde{B}_{\delta/2}
(({\phi}^{*},t^{*}),\widetilde{\gamma}))\subseteq
\widetilde{B}_{\delta}(({\phi}^{*},t^{*}),\widetilde{C}\widetilde{\gamma}),$$
for all $\widetilde{\gamma}<\widetilde{\gamma}_1.$ Then there exists
$\widetilde{\gamma}_2$ such that if
$$|(\phi,t)-({\phi}^{*},t^{*})|_{\mathcal{A}_{\delta/2}\times \mathbb{C}}<\widetilde{C}\widetilde{\gamma}_2$$
then $\phi\circ q_t$ is defined in $\overline{D}_{2\varepsilon}.$ In
particular for each $\widetilde{\gamma}\leq \widetilde{\gamma}_2$ we
have defined the following composition transformations
$$L:\widetilde{B}_{\delta/2}(({\phi}^{*},t^{*}),\widetilde{\gamma})\rightarrow \mathcal{U}_{D_{2\varepsilon}},$$
and
$$L:\widetilde{B}_{\delta/2}(({\phi}^{*},t^{*}),\widetilde{C}\widetilde{\gamma})\rightarrow
\mathcal{U}_{D_{2\varepsilon}}.$$ Let $\widetilde{\gamma}_3>0$ such that for all
$\widetilde{\gamma}\leq \widetilde{\gamma}_3$ we have
$$L(\widetilde{B}_{\delta/2}(({\phi}^{*},t^{*}),\widetilde{\gamma}))
\subseteq B_{2\varepsilon}({\phi}^{*}\circ q_{t^*},\gamma_1).$$ Choose
$\widetilde{\gamma}\leq \min
\{\widetilde{\gamma}_1,\widetilde{\gamma}_2,\widetilde{\gamma}_3\}.$
So we stated the following results.
\begin{proposition}
The following diagram commutes.
\vspace{1em}
$$\xymatrix@R=1cm{\widetilde{B}_{\delta/2}(({\phi}^{*},t^{*}),\widetilde{\gamma})
\ar[r]^{\mathcal{\widetilde{\widetilde{R}}}} \ar[d]^{L}
\ar@/_40pt/[dd]_{L=k\circ L } \ar@/^30pt/[rr]_{j\circ
\mathcal{\widetilde{\widetilde{R}}}=\mathcal{\widetilde{R}}}&
\widetilde{B}_{\delta}(({\phi}^{*},t^{*}),\widetilde{C}\widetilde{\gamma})
\ar[r]^{j} &
\widetilde{B}_{\delta/2}(({\phi}^{*},t^{*}),\widetilde{C}\widetilde{\gamma})
\ar[d]_{L} \ar@/^40pt/[dd]^{L=k\circ L }\\
 B_{2\varepsilon}({\phi}^{*}\circ q_{t^*},\gamma_1)
 \ar[d]^{k} & & \mathcal{U}_{D_{2\varepsilon}}
 \ar[d]_{k} \\
 B_{\varepsilon/2}({\phi}^{*}\circ q_{t^*},\gamma_1)
 \ar[r]^{\mathcal{\hat{R}}} \ar@/_25pt/[rr]^{i\circ \mathcal{\hat{R}}=\mathcal{R}} &
\mathcal{U}_{D_{\varepsilon}}
 \ar[r]^{i} & \mathcal{U}_{D_{\varepsilon/2}}
 }$$
\vspace{1em}\\In particular we have that $\mathcal{R}\circ L=L\circ
\mathcal{\widetilde{R}}$ on $\widetilde{B}_{\delta/2}(({\phi}^{*},t^{*}),\widetilde{\gamma}).$
\end{proposition}

For the respective tangent spaces we obtain
\begin{proposition}
The following diagram commutes.
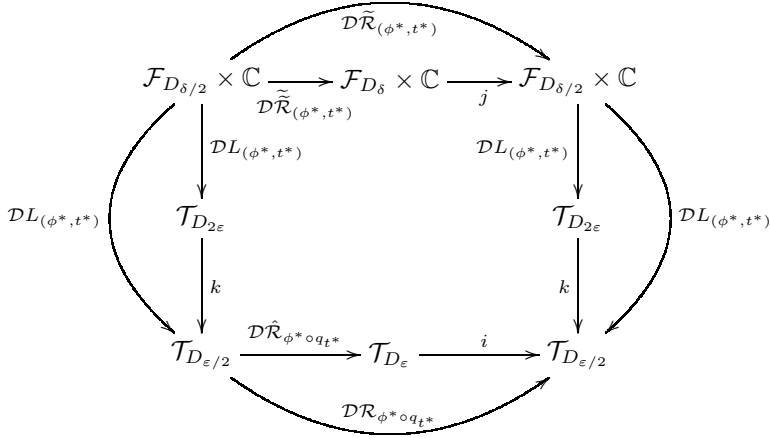
\vspace{1em}
$$\xymatrix@R=1.25cm{\mathcal{F}_{D_{\delta/2}}\times
\mathbb{C}
\ar[r]_(0.55){\D\mathcal{\widetilde{\widetilde{R}}}_{(\phi^*,t^*)}}
\ar[d]^{\D L_{(\phi^*,t^*)}} \ar@/_35pt/[dd]_{ \D L_{(\phi^*,t^*)}}
\ar@/^30pt/[rr]_{ \D\mathcal{\widetilde{R}}_{(\phi^*,t^*)}}&
\mathcal{F}_{D_{\delta}}\times \mathbb{C}\ar[r]_{j} &
\mathcal{F}_{D_{\delta/2}}\times \mathbb{C}
\ar[d]_{\D L_{(\phi^*,t^*)}} \ar@/^35pt/[dd]^{ \D L_{(\phi^*,t^*)}}\\
 \mathcal{T}_{D_{2\varepsilon}}
 \ar[d]^{k} & & \mathcal{T}_{D_{2\varepsilon}}
 \ar[d]_{k} \\
\mathcal{T}_{D_{\varepsilon/2}}
 \ar[r]^{\D\mathcal{\hat{R}}_{\phi^*\circ q_{t^*}}} \ar@/_30pt/[rr]^{\D\mathcal{R}_{\phi^*\circ q_{t^*}}} &
\mathcal{T}_{D_{\varepsilon}}
 \ar[r]^{i} & \mathcal{T}_{D_{\varepsilon/2}}
 }$$
\vspace{1em}
In particular we have that
$\D
\mathcal{R}_{L({\phi}^{*},t^{*})}\cdot\D L_{({\phi}^{*},t^{*})}=\D
L_{({\phi}^{*},t^{*})}\cdot\D
\mathcal{\widetilde{R}}_{({\phi}^{*},t^{*})}$ on $\mathcal{F}_{D_{\delta/2}}\times
\mathbb{C}.$
\end{proposition}
Notice that, by
Remark~\ref{'relacL'}, $L({\phi}^{*},t^{*})$ is fixed point of
the operator $\mathcal{R}.$ An important relation between the
operators $\mathcal{R}$ and $\mathcal{\widetilde{R}}$ is the
following result.

\begin{proposition}\label{'spectro'}
Let $({\phi}^{*},t^{*})\in\mathcal{A}_{D_{\delta/2}}\times(0,1]$ be the
fixed point of the operator $\widetilde{\mathcal{R}}$ and
$L({\phi}^{*},t^{*})$ the corresponding fixed point of
$\mathcal{R}.$ Then
$$\sigma(\D\widetilde{\mathcal{R}}_{({\phi}^{*},t^{*})})=
\sigma(\D\mathcal{R}_{L({\phi}^{*},t^{*})}).$$
\end{proposition}
\begin{proof} Denote $f_*=({\phi}^{*},t^{*}).$ Let $\lambda\neq 0$ be an
eigenvalue of $\D\widetilde{\mathcal{R}}_{f_*}$ with eigenvector
$v\neq 0,$ this is
$$\D\widetilde{\mathcal{R}}_{f_*}\,v=\lambda v.$$
By the relation of composition on the tangent spaces above we
have
$$\D \mathcal{R}_{L(f_*)}\cdot\D L_{f_*}=\D L_{f_*}\cdot\D \mathcal{\widetilde{R}}_{f_*}.$$
Then
\begin{eqnarray*}
\D \mathcal{R}_{L(f_*)}\cdot\D L_{f_*}\, (v)&=&\D L_{f_*}\cdot\D \mathcal{\widetilde{R}}_{f_*}\,(v)\\
\D \mathcal{R}_{L(f_*)}(\D L_{f_*\,}\,v)&=&\lambda(\D L_{f_*}\, v)
\end{eqnarray*}
so $\lambda$ is an eigenvalue of $\D\mathcal{R}_{L(f_*)}$ with $\D
L_{f_*}\, v\neq 0$ from Proposition~\ref{'propinj'}. Finally from
compactness of the operators $\D\widetilde{\mathcal{R}}_{f_*}$ and
$\D\mathcal{R}_{L(f_*)}$ it follows that
$\sigma(\D\widetilde{\mathcal{R}}_{f_*})\subset~\sigma(\D\mathcal{R}_{L(f_*)}).$

  Let $w \in\mathcal{T}_{D_{\varepsilon/2}}$.  Note that  $w(x)=\sum_{i}
a_i x^i,$ with $a_i=0$ if  $ 2r\nmid i$. Let $r_v\geq \epsilon/2$ be the convergence radius of this series.  Define the Taylor series  $\psi(x)=\sum_{i} a_{2ri}
x^i.$ The convergence radius for $\psi,$ denoted by $r_{\psi},$ is at least $(\varepsilon/2)^{2r}$ since
\begin{eqnarray*}
r_{\psi}&= & \frac{1}{\limsup_i \sqrt[i]{|a_{2ri}|}}\\ &\geq&
{\left(\frac{1}{\limsup_i \sqrt[2ri]{|a_{2ri}|}} \right) }^{2r}\\ &=& r_{w}^{2r}\\ & \geq& \big( {\varepsilon/2}\big)^{2r},
\end{eqnarray*}
so $\psi$ is well defined in a neighborhood of $0$. Note that $w(x)=\psi(x^{2r})$ for $x$ in a neighborhood of $0$. Indeed it is easy to see that $\psi$ is
defined in a $\theta(2r,\varepsilon)$-stadium of the interval
$[0,1]$ because $w$ is defined in a neighborhood of
the interval $[-1,1].$ Take $\mu(2r,\varepsilon,t)< \delta/2$ such that the
image of the $\mu(2r,\varepsilon,t)$-stadium
$D_{\mu(2r,\varepsilon,t)}$ by $A_{t^{*}}^{-1}$ is contained in
the $\theta(2r,\varepsilon)$-stadium of the interval $[0,1].$ Now
define
$$F:\mathcal{T}_{D_{\varepsilon/2}}\rightarrow
\mathcal{F}_{D_{\mu(2r,\varepsilon,t)}}\times \mathbb{C}$$ by
\begin{eqnarray}
F(w)= (\omega,b),\label{'aplicF'}
\end{eqnarray} where
$$\omega(y)=\psi\circ A_{t^{*}}^{-1}(y)-D\phi^{*}(y)
\frac{\psi\circ A_{t^{*}}^{-1}(1)}{2D\phi^{*}(1)}(1+y)$$ and
$$b=\frac{\psi\circ A_{t^{*}}^{-1}(1)t^{*}}{2D\phi^{*}(1)}.$$
Note that
$$ \mathcal{D}L_{f^\star} F(w)=w, $$
where this equality holds in a complex neighborhood of $[-1,1]$.
Let $\lambda\neq 0$ be an eigenvalue of $\D\mathcal{R}_{L(f_*)}$ with
eigenvector $w\neq 0,$ that is
$$\D\mathcal{R}_{L(f_*)}\,w=\lambda w.$$
 on $[-1,1]$. Then we have
$$\D L_{f_*}\cdot\D \widetilde{\mathcal{R}}_{f_*}\,F(w)=\D \mathcal{R}_{L(f_*)}\cdot\D L_{f_*}\,F(w)$$
on a neighborhood of $[-1,1]$. So
$$\D L_{f_*}\cdot\D \widetilde{\mathcal{R}}_{f_*}\,F(w)=\D \mathcal{R}_{L(f_*)}w= \lambda w = \lambda \mathcal{D}L_{f^\star} F(w)$$
on a neighborhood of $[-1,1]$. By the injectivity of $\D L_{f_*}$ it follows that
$$\D \widetilde{\mathcal{R}}_{f_*}\,F(w)= \lambda F(w)$$
in a neighborhood of $[-1,1]$. Note that  $F(w) \in \mathcal{F}_{D_{\eta}}\times \mathbb{C}$, for some $\eta$.   It remains to show that  $F(w)=(\omega,b)$ belongs to
$\mathcal{F}_{D_{\delta/2}}\times \mathbb{C}$. Indeed by the complex bounds there exists $N_0$ such that

$$\mathcal{D} {\tilde{\mathcal{R}}}^{N_0}_{f_\star} \colon \mathcal{F}_{D_{\eta}} \times \mathbb{C}\rightarrow \mathcal{F}_{D_{\delta/2}}\times \mathbb{C}$$
is well defined. In particular $\mathcal{D} {\tilde{\mathcal{R}}}^{N_0}_{f_\star}F(w)=\lambda^{N_0}F(w)$ belongs to $\mathcal{F}_{D_{\delta/2}}\times \mathbb{C}$, so $F(w)$ belongs to $\mathcal{F}_{D_{\delta/2}}\times \mathbb{C}$. We conclude that
$\sigma(\D\mathcal{R}_{L(f_\star)})\subset \sigma(\D\widetilde{\mathcal{R}}_{f_*})$.
\end{proof}

We denote by $V_{\beta}$ and $\widetilde{V}_{\lambda}$ the respective
eigenspaces of the eigenvalues  $\beta \in
\sigma(\D\mathcal{R}_{L({\phi}^{*},t^{*})})$  and $\lambda \in
\sigma(\D\widetilde{\mathcal{R}}_{({\phi}^{*},t^{*})}),$ this is
$$V_{\beta}=Ker(\D\mathcal{R}_{L({\phi}^{*},t^{*})}-\beta Id)=
\{v:(\D\mathcal{R}_{L({\phi}^{*},t^{*})}-\beta Id)v=0\}$$ e
$$\widetilde{V}_{\lambda}=Ker(\D\widetilde{\mathcal{R}}_{({\phi}^{*},t^{*})}-\lambda Id )=
\{(\omega,t):(\D\widetilde{\mathcal{R}}_{({\phi}^{*},t^{*})}-\lambda
Id )(\omega,t)=0\}.$$ Theses eigenspaces does not depending of the domains of definition of the maps since to apply the renormalization
operator ($\mathcal{R}$ or $\mathcal{\widetilde{R}}$) are holomorphically improving operators.

\begin{proposition}\label{'dimautospace'}
Let $({\phi}^{*},t^{*})\in\mathcal{A}_{D_{\delta/2}}\times(0,1]$ be the
fixed point of the operator $\widetilde{\mathcal{R}}. $ If $\lambda
\in \sigma(\D\widetilde{\mathcal{R}}_{({\phi}^{*},t^{*})})\setminus
\{0\}$ then $\dim V_{\lambda}=\dim \widetilde{V}_{\lambda}.$
\end{proposition}
\begin{proof} We have that $V_{\lambda}$ and
$\widetilde{V}_{\lambda}$ are finite dimensional subspaces. Define the
continuous map
$$T: \widetilde{V}_{\lambda}\rightarrow V_{\lambda}$$ by
$T(v)=\D L_{({\phi}^{*},t^{*})}(v).$ Since $\D
L_{({\phi}^{*},t^{*})}$ is injective we have that $\dim Ker(T)=0.$
Then $\dim \widetilde{V}_{\lambda}= \dim Ker(T)+\dim Im(T)\leq \dim
V_{\lambda}.$ Also we can define the continuous injective map
$$\widetilde{T}:V_{\lambda} \rightarrow \widetilde{V}_{\lambda}$$
by $\widetilde{T}(\widetilde{w})=F(\widetilde{w}),$ where $F$ is
defined by the expression(~\ref{'aplicF'}). Then $\dim V_{\lambda}\leq
\dim \widetilde{V}_{\lambda}.$
\end{proof}

\section{Proof of the main results}
The results obtained in the Subsection~\ref{'operRenovo'} (the
Propositions~\ref{'spectro'} and~\ref{'dimautospace'})
and the Theorem~\ref{'hyperbolicity'} \cite{E-W-P} were crucial to
establish a result of hyperbolicity for the new renormalization
operator $\mathcal{\widetilde{R}}_{2r}$ which is analog to the
Theorem~\ref{'hyperbolicity'} \cite{E-W-P}, for the usual
renormalization $\mathcal{R}_{2r}.$

\begin{proposition}\label{'newopersuave'}
The transformation
$(\phi,t,\alpha)\rightarrow\mathcal{\widetilde{R}}_{\alpha}(\phi,t)$
is complex analytic in the variables $(\phi,t)$ and real analytic in
the variable $(\phi,t,\alpha).$
\end{proposition}

\begin{theorem}
Let $\alpha=2r\in 2\mathbb{N}$ be the critical exponent. There exists a positive
number $N$ such that the operator $\mathcal{\widetilde{R}}_{2r},$ as
defined above, has a unique unimodal fixed point
$(\phi^*_{2r},t^*_{2r}).$ Furthermore $(\phi^*_{2r},t^*_{2r})$ is
hyperbolic with codimension one stable manifold.
\end{theorem}
\begin{proof}
By Martens~\cite{MM}, there exists a fixed point $(\phi^*,t^*)$ to the operator
$\mathcal{\widetilde{R}}_{2r}.$
Since
$$L\circ\mathcal{\widetilde{R}}_{2r}(\phi^*,t^*)=\mathcal{R}_{2r}\circ
L(\phi^*,t^*),$$ see Remark~\ref{'relacL'}, then
$\phi^*\circ q_{t^*}$ is a fixed point to the operator $\mathcal{R}_{2r}.$ On the other hand Sullivan~\cite{Su} and
Theorem~\ref{'hyperbolicity'}\cite{E-W-P} imply that  the operator
$\mathcal{R}_{2r}$ has a hyperbolic fixed point $f^*=\phi^*\circ
q_{t^*} \in B(f^*,\gamma_1)$ with codimesion one stable manifold.
By Proposition~\ref{'spectro'} and Proposition~\ref{'dimautospace'} we obtain that the fixed point $(\phi^*,t^*)$ of the operator $\mathcal{\widetilde{R}}_{2r}$
is hyperbolic with codimension one stable maniflod. The uniqueness of the fixed point follows
from Proposition~\ref{'unipointren'}.
\end{proof}

\begin{bew}{of Theorem~\ref{'fixedhyp'}}
Define the operator $$F((\phi,t),\alpha)= \widetilde{\mathcal{R}}_{\alpha}(\phi,t) - (\phi,t).  $$
From Proposition~\ref{'newopersuave'}  the
operator $F$ is complex analytic in the variables $(\phi,t)$ and
real analytic in the variables $(\phi,t,\alpha).$ We have
$$\mathcal{D}\widetilde{\mathcal{R}}_{2r}(\phi_{2r}^*,t_{2r}^*)-Id$$ is
invertible because $(\phi_{2r}^*,t_{2r}^*)$ is the hyperbolic fixed
point of $\widetilde{\mathcal{R}}_{2r}.$ So we can conclude the proof applying the  Implicit
Function Theorem.
\end{bew}

\begin{bew}{of  Theorem~\ref{uniqueness}} Since $(\phi_{2r}^*,t_{2r}^*)$ is a hyperbolic fixed point,
there exists a neighborhood $V_1=B((\phi_{2r}^*,t_{2r}^*),\eta)$ of
$(\phi_{2r}^*,t_{2r}^*)$ such that for $\alpha\sim 2r,$ there exists
an unique fixed point of $\mathcal{\widetilde{R}}_{\alpha}$ in
$V_1.$ Therefore it only remains to verify that, for $\alpha\sim 2r,$ all
hyperbolic fixed points $(\phi_{\alpha}^*,t_{\alpha}^*)$ of
$\mathcal{\widetilde{R}}_{\alpha}$ belong to $V_1$. In fact, we
suppose that there exists a sequence $\alpha_s\rightarrow 2r,$ where
$s\rightarrow \infty,$ and fixed points
$(\phi_{{\alpha}_s}^*,t_{{\alpha}_s}^*)$ of
$\mathcal{\widetilde{R}}_{{\alpha}_s}$ such that either
\begin{eqnarray}
|\phi_{\alpha_s}^*-\phi_{2r}^*|_{D_{\delta}}\geq \eta \mbox{ or }
|t_{\alpha_s}^*-t^*_{2r}|\geq\eta.\label{'negunit'}
\end{eqnarray}
From Corollary~\ref{'restatecomplex'}
$$(\phi_{\alpha_s}^*,t_{\alpha_s}^*)\in
\mathcal{H}_{\alpha_s}(C_0,\delta_0,\infty).$$ By definition of
$\widetilde{\widetilde{R}}_{\alpha_s}$ we have that
$$(\phi_{\alpha_{s}}^*,t_{\alpha_{s}}^*)=\mathcal{\widetilde{\widetilde{R}}}_{\alpha_{s}}(\phi_{\alpha_{s}}^*,t_{\alpha_{s}}^*)$$
belongs to $\mathcal{H}_{\alpha_s}(C_0,2\delta_0,\infty).$ Then we have
$(\phi_{\alpha_s}^*,t_{\alpha_s}^*)$ is a pre-compact family on
$\mathcal{A}_{D_{\delta_0}}\times \mathbb{C},$ in particular there
is a subsequence $(\phi_{\alpha_{s_i}}^*,t_{\alpha_{s_i}}^*)$
converging to $(\phi,t)\in \mathcal{H}_{2r}(C_0,\delta_0,\infty).$
Since
\begin{eqnarray*}
\mathcal{\widetilde{R}}_{\alpha_{s_i}}(\phi_{\alpha_{s_i}}^*,t_{\alpha_{s_i}}^*)=(\phi_{\alpha_{s_i}}^*,t_{\alpha_{s_i}}^*)
\end{eqnarray*} taking $s_i\rightarrow 2r$, we conclude that
$(\phi,t)\in \mathcal{H}_{2r}(C_0,\delta_0,\infty)$ is fixed point
of the operator $\mathcal{\widetilde{R}}_{2r}.$ Then by the
uniqueness of the fixed point to $\alpha=2r$ we have
$$(\phi,t)=(\phi_{2r}^*,t_{2r}^*).$$
This leads to a contradiction with~(\ref{'negunit'}).
\end{bew}

\begin{theorem}[Convergence]\label{'convergencia'}
If $(\phi,t)\in B_{\delta}((\phi^*, t^*),\widetilde{\gamma})$ is an
unimodal pair, infinitely renormalizable with combinatorics
$\sigma,$ then we have $\mathcal{\widetilde{R}}^i_{2r}(\phi,t)\in
B_{\delta}((\phi^*, t^*),\widetilde{\gamma}),$ for all $i$ large
enough and
$$\mathcal{\widetilde{R}}_{2r}^i(\phi, t)\rightarrow_i
(\phi^*, t^*)$$ in $\mathcal{A}_{D_{\delta}}\times\mathbb{C}.$
\end{theorem}
\begin{proof} Observe that by the complex bounds, we have
$$\mathcal{\widetilde{R}}^i_{2r}(\phi,t) \in
\mathcal{H}_{2r}(C_0.\delta_0,\infty),$$ for all $i$ large enough.
Suppose that the statement of the theorem is false. So there exists
$\eta>0$ and a pair $(\phi,t)$ infinitely renormalizable with
critical exponent $2r$ such that
  \begin{eqnarray}
   |\mathcal{\widetilde{R}}_{2r}^{i_j}(\phi,t)-(\phi^*,t^*)|_{\mathcal{A}_
   {D_{\delta}}\times \mathbb{C}}>\eta,\label{'contrachyp'}
  \end{eqnarray}
where $i_j\rightarrow_j \infty$. We have the first component of each par in the family
$\{\mathcal{\widetilde{\widetilde{R}}}_{2r}\circ\mathcal{\widetilde{R}}_{2r}^{i_j-1}(\phi,t)\}_{j}$
is defined and univalent in $\overline{D}_{2\delta}.$ Then this familyis a pre-compact in $\mathcal{A}_{D_{\delta}}\times \mathbb{C},$
in particular there exists a convergence subsequence
$$\mathcal{\widetilde{R}}_{2r}^{i_j}(\phi, t)\rightarrow (\widetilde{\phi},
\widetilde{t}).$$
The composition operator L satisfies
$$L\circ\widetilde{R}_{2r}= \mathcal{R}_{2r}\circ L,$$ so it follows that
$$\mathcal{R}_{2r}^{i_j}(\phi\circ q_{t})=\mathcal{R}_{2r}^{i_j}\circ
L(\phi,t)\rightarrow_j L(\widetilde{\phi},
\widetilde{t})=\widetilde{\phi}\circ q_{\widetilde{t}}. $$ By
Sullivan~\cite{Su} ( also see~\cite{McM96}) the operator
$\mathcal{R}_{2r}$ has an unique fixed point $\phi^*\circ q_{t^*}$ and
furthermore
$$\mathcal{R}^{i}_{2r}(\phi\circ q_t)\rightarrow_i \phi^*\circ
q_{t^*},$$ for all $\phi\circ q_t$ infinitely renormalizable. So
$\phi^*\circ q_{t^*}=\widetilde{\phi}\circ q_{\widetilde{t}}.$ By
Proposition~\ref{'unipointren'} we have that
$\widetilde{\phi}=\phi^*$ e $\widetilde{t}=t^*.$ This leads to a
contradiction with Eq. (\ref{'contrachyp'}).
\end{proof}

In the proof of the following result we use many tools and concepts of complex dynamic, as polynomial-like maps,
quasiconformal maps and Sullivan's pullback argument see~\cite{W-V},~\cite{E-W},~\cite{McM94}
and~\cite{McM96}.

\begin{theorem}[Equicontinuity]\label{'contuniforme'}
Let $(\phi,t) \in \mathcal{H}_{2r}(C,\delta_0,\infty)$. Then there
exists $i_0$ such that for all $\tilde{\gamma} > 0$ there exists
$\tilde{\eta}
> 0$  with the following property. If $ (\psi,v) \in
\mathcal{H}_{2r}(C,\delta_0,\infty)$ satisfies
$$|(\psi,v)-(\phi,t)|_{\mathcal{A}_{D_\delta}\times \mathbb{C}} < \tilde{\eta}$$
then
$$|\tilde{\mathcal{R}}^i_{2r}(\psi, v) - \tilde{\mathcal{R}}^i_{2r}(\phi, t)|_{\mathcal{A}_{D_\delta}\times \mathbb{C}} < \tilde{\gamma}$$
for all $i \geq i_0$.
\end{theorem}
\begin{proof}
Firstly we claim that there exists $i_0$ such that for all
$\gamma> 0$ there exists $\eta> 0$ with the following property. If
$$ (\psi,v) \in \mathcal{H}_{2r}(C,\delta_0,\infty)$$
satisfies
$$|(\phi,t)-(\psi,v)|_{\mathcal{A}_{D_{\delta_0}}\times \mathbb{C}} < \eta$$
then
$$|\mathcal{R}^i_{2r}(\psi\circ q_v) - \mathcal{R}^i_{2r}(\phi\circ q_t)|_{D_\varepsilon} < \gamma$$
for all $i \geq i_0$. In fact, since a map $f:=\phi\circ q_t$ is a
unimodal analytic map in a neighborhood $V_0$ of the interval
$[-1,1]$, with critical point of order $2r$, by the complex bounds
of Sullivan~\cite{Su}, for $i_0$ big enough there exists a
polynomial-like extension
$$\mathcal{R}_{2r}^{i_0}f\colon U_f \rightarrow U,$$
where $[-1,1]\subset U_f \Subset U$. If $(\psi,v) \in
\mathcal{H}_{2r}(C,\delta_0,\infty)$ is close to $(\phi,t)$, then
$g=\phi\circ q_v$ is close to $f$, then $\mathcal{R}_{2r}^{i_0}g$
has a polynomial-like extension
$$\mathcal{R}_{2r}^{i_0}g\colon U_g \rightarrow U.$$ Here we can choose $U$ such that the disc $U_g$
is moving holomorphically with respect to $g$. In particular, by the
theory of holomorphic motions, the pullback argument of Sullivan
(see~\cite{W-V}) and the no-existence of invariant lines fields with
support on the filled Julia set of $f$, there exists  quasiconformal homeomorphisms
$h_g\colon \mathbb{C} \rightarrow \mathbb{C}$ such
that
\begin{equation} \label{conj}
(\mathcal{R}^{i_0} g)\circ h_g (x)= h_g \circ (\mathcal{R}^{i_0}f)(x)
\end{equation}
for all $x \in U_f$. Furthermore the quasiconformality $Q(g)$ of
$h_g$ satisfies
$$Q(\psi\circ q_v)\rightarrow 1  \text{ when }  (\psi,v)\rightarrow (\phi,t).$$
Notice that since all the following  renormalizations of $f$ and $g$ are
conjugated by rescalings of the same conjugation $h_g$, then the
quasiconformality of those conjugations are bounded by $Q(\psi\circ
q_v)$. Then since all conjugacies between the ith-renormalization of $f$ and $g$, with $i>i_{0},$ fix -1 and 1 , it follows that theses conjugacies converges uniformly to the identity on compact subsets of $\mathbb{C}$ when $(\psi,v)$ converges to
$(\phi,t)$.

\noindent Since we proved complex bounds, the sequences
$\mathcal{R}^i_{2r}f$ e $\mathcal{R}^i_{2r}g$ are bounded in
$\mathcal{U}_{D_{\varepsilon}}$, it follows from   Eq. (\ref{conj}) that
if $(\psi,v)$ is close enough to $(\phi,t)$ then
$$|\mathcal{R}^i_{2r}(\psi\circ q_v) - \mathcal{R}^i_{2r}(\phi\circ q_t)|_{D_\varepsilon} < \gamma$$
for $i\geq i_0$. So we have proved the claim.

\noindent Suppose by
contradiction that there exists a sequence $(\psi_j, v_j)$ such that
$$|(\phi,t)-(\psi_j,v_j)|_{\mathcal{A}_{D_\delta}\times \mathbb{C}} \rightarrow_j 0$$
but
\begin{equation}\label{noconv} |\tilde{\mathcal{R}}^{i_j}_{2r}(\psi_j, v_j) -
\tilde{\mathcal{R}}^{i_j}_{2r}(\phi,
t)|_{\mathcal{A}_{D_\delta}\times \mathbb{C}} >
\tilde{\gamma},
\end{equation}
with $i_j\geq i_0$ and
$i_j\rightarrow_j \infty$.\\Notice, by the complex and real bounds,
that
$$\tilde{\mathcal{R}}^{i_j}_{2r}(\psi_j, v_j)=(\psi_{j,i_j},v_{j,i_j}),$$
where $\psi_{j,i_j}$ is univalent in $D_{2\delta}$ and fix $1$ and
$-1,$ and moreover
$$0 < \inf_{j}v_{j,i_j}\leq \sup_{j}v_{j,i_j} < 1.$$
so we can suppose that the second coordinate of those pairs converge
for same $v_\infty \in (0,1)$. As $\psi_{i,i_j}$ is univalent in
$D_{2\delta}$, taking a subsequence, if necessary, we can suppose
that ${\psi_{i,i_j}}$ converges for some univalent map $\psi_\infty$
in $\overline{D}_\delta$. In particular
$$\mathcal{R}^{i_j}_{2r}(\psi_j\circ q_{v_j})=L(\tilde{\mathcal{R}}^{i_j}_{2r}(\psi_j, v_j))\rightarrow_j \psi_\infty\circ q_{v_\infty}.$$
On the other hand, by the claim that we proved at the beginning of
the proof and the Theorem~\ref{'convergencia'} we obtain
$$\mathcal{R}^{i_j}_{2r}(\psi_j\circ q_{v_j}) \rightarrow_j \phi^*\circ q_{t^*}.$$
Then $\psi_\infty\circ q_{v_{\infty}}= \phi^*\circ q_{t^*}$. By
Proposition \ref{'unipointren'} we conclude that $$
(\psi_\infty,v_\infty)=(\phi^*,t^*).$$ So
\begin{equation} \label{igual}
|\tilde{\mathcal{R}}^{i_j}_{2r}(\psi_j, v_j)- (\phi^*, t^*)|_{\mathcal{A}_{D_\delta}
\times \mathbb{C}}\rightarrow_j 0.
\end{equation}
But Eq.~(\ref{noconv}) and Theorem~\ref{'convergencia'} imply that
$$|\tilde{\mathcal{R}}^{i_j}_{2r}(\psi_j, v_j) - (\phi^*, t^*)|_{\mathcal{A}_{D_\delta}\times \mathbb{C}} >  \tilde{\gamma}/2$$
for $j$ large enough. This contradicts Eq. (\ref{igual}).
\end{proof}

Let $a>0$ be a constant. We denote by
$B_{a}^{\infty}(\phi_{\alpha}^*,t_{\alpha}^*)$ the set of infinitely
renormalizable pairs $(\phi,t)$ by
$\mathcal{\widetilde{R}}_{\alpha}$ such that
$$|(\phi,t)-(\phi_{\alpha}^*,t_{\alpha}^*)|_{\mathcal{A}_{D_{\delta}}\times
\mathbb{C}}<a.$$

\begin{theorem}\label{'manifold'}
For all $\gamma>0$ there exists $N_2$ such that for all $\alpha\sim
2r$ we have
\begin{eqnarray}
 \mathcal{\widetilde{R}}^{N_2}_{\alpha}(\mathcal{H}_{\alpha}
 (C_0,\delta_0,\infty))\subset
 B_{\gamma}^{\infty}(\phi_{2r}^*,t_{2r}^*).\label{'afirmteo'}
\end{eqnarray}
\end{theorem}
\begin{proof} Let $\gamma>0.$ We claim there exists $M_2$ such that
\begin{eqnarray}
 \mathcal{\widetilde{R}}^{M_2}_{2r}(\mathcal{H}_{2r}(C_0,\delta_0,\infty))
 \subset B_{\gamma/2}^{\infty}(\phi_{2r}^*,t_{2r}^*).
\end{eqnarray}
In fact, we suppose that there exists a sequence
$(\phi_i,t_i)\in\mathcal{H}_{2r}(C_0,\delta_0,\infty) $ such that
for a subsequence $j_i\rightarrow \infty $ we have
\begin{eqnarray}
|\mathcal{\widetilde{R}}^{j_i}_{2r}
(\phi_i,t_i)-(\phi_{2r}^*,t_{2r}^*)|_{\mathcal{A}_{D_{\delta}}\times
\mathbb{C}}>\gamma/2.\label{'afirm2'}
\end{eqnarray}
By Corollary~\ref{'restatecomplex'} we obtain that
$\mathcal{\widetilde{\widetilde{R}}}_{2r}(\phi_i,t_i)\in\mathcal{H}_{2r}(C_0,2\delta_0,\infty),$
in particular ${\mathcal{\widetilde{\widetilde{R}}}_{2r}(\phi_i,t_i)}_{i}$ is a pre-compact family in
$\mathcal{A}_{D_{\delta_0}}\times \mathbb{C}.$ Taking a subsequence, if necessary, we can assume without loss of generality  that
$$\mathcal{\widetilde{R}}_{2r}(\phi_i,t_i)=j\circ
\mathcal{\widetilde{\widetilde{R}}}_{2r}(\phi_i,t_i)\rightarrow
(\phi,t),$$ where
$(\phi,t)\in\mathcal{H}_{2r}(C_0,\delta_0,\infty).$ By
Theorem~\ref{'convergencia'} we have for $\gamma>0$ there exists
$k_0>0$ such that for all $k>k_0$
$$|\mathcal{\widetilde{R}}_{2r}^k(\phi_i, t_i)-(\phi_{2r}^*, t_{2r}^*)|_{\mathcal{A}_{D_{\delta}}\times \mathbb{C}}<\frac{\gamma}{8}.$$
And using the Theorem~\ref{'contuniforme'} there exists $i_0>0$ and
$k_1>0$ such that for all $i>i_0$ and $k>k_1$ we have that
$$|\mathcal{\widetilde{R}}_{2r}^k(\mathcal{\widetilde{R}}_{2r}(\phi_i, t_i))-\mathcal{\widetilde{R}}_{2r}^k(\phi, t)|
_{\mathcal{A}_{D_{\delta}}\times \mathbb{C}}<\frac{\gamma}{8}.$$
Then for $k>\max \{k_0,k_1\}$ we obtain
$$|\mathcal{\widetilde{R}}_{2r}^{k+1}(\phi_i, t_i)-(\phi_{2r}^*, t_{2r}^*)|_{\mathcal{A}_{D_{\delta}}\times \mathbb{C}}<\frac{\gamma}{4}$$
this contradicts Eq.~(\ref{'afirm2'}). This proves the claim.

\noindent Take $N_2=M_2+1$.  We claim the $N_{2}$ satisfies Eq.~(\ref{'afirmteo'}). Otherwise we could find
a sequence $\alpha_s\rightarrow 2r$ and
$(\phi_{s},t_{s})\in \mathcal{H}_{\alpha_s}(C_0,\delta_0,\infty)$
such that
\begin{eqnarray}
|\mathcal{\widetilde{R}}^{M_2+1}_{\alpha_s}
(\phi_{s},t_{s})-(\phi_{2r}^*,t_{2r}^*)|_{\mathcal{A}_{D_{\delta}}\times
\mathbb{C}}>\gamma.\label{'zero'}
\end{eqnarray}
We have that the first component of the pairs from the family
$\{\mathcal{\widetilde{\widetilde{R}}}_{\alpha_s}(\phi_{s},t_{s})\}_{s}$
has a complex univalent extension to $D_{2\delta_0}.$ So this family
is pre-compact on $\mathcal{A}_{D_{\delta_0}}\times \mathbb{C}.$ In
particular there exists a subsequence
$\mathcal{\widetilde{\widetilde{R}}}_{\alpha_{s_i}}(\phi_{s_i},t_{s_i})$
on $\mathcal{A}_{D_{\delta_0}}\times \mathbb{C}$ that converges to
some $(\widetilde{\phi},\widetilde{t})\in
\mathcal{H}_{2r}(C_0,\delta_0,\infty).$ From the above we have
$$\mathcal{\widetilde{R}}^{M_2}_{2r}(\widetilde{\phi},\widetilde{t})\in B_{\gamma/2}^{\infty}(\phi_{2r}^*,t_{2r}^*).$$
On the other hand
$$\mathcal{\widetilde{R}}^{M_2+1}_{\alpha_{s_i}}(\phi_{s_i},t_{s_i})\rightarrow
\mathcal{\widetilde{R}}^{M_2}_{2r}(\widetilde{\phi},\widetilde{t})$$
in $\mathcal{A}_{D_{\delta_0}}\times \mathbb{C}.$ So for $i$ large
enough we have
\begin{eqnarray*}
|\mathcal{\widetilde{R}}^{M_2+1}_{\alpha_{s_i}}(\phi_{s_i},t_{s_i})-(\phi_{2r}^*,t_{2r}^*)|_{\mathcal{A}_{D_{\delta}}\times
\mathbb{C}}<\gamma.
\end{eqnarray*}
This leads to a contradiction with the Eq.~(\ref{'zero'}).
\end{proof}

\begin{bew}{of the Theorem~\ref{variedade}}
Let $V_1$ be a neighborhood satisfying the Eq. (\ref{stable}).
Choose $\gamma>0$ such that $B_{\gamma}(\phi_{2r}^*,t_{2r}^*)$ is
contained in $V_1.$\\\emph{Claim I.}There exists $\gamma_1$ such that
for all $\alpha\sim 2r$ we have
$$B_{\gamma_1}^{\infty}(\phi_{2r}^*,t_{2r}^*)\subset
W^s_{V_1}.$$ In fact, let $N_2$ be as in the Theorem~\ref{'manifold'}.
Choose $\gamma_1<\gamma$ small enough such that for all
$i=1,..,N_2,$
$$ \mathcal{\widetilde{R}}^{i}B_{\gamma_1}^{\infty}(\phi_{2r}^*,t_{2r}^*)\subset
 B_{\gamma}^{\infty}(\phi_{2r}^*,t_{2r}^*)\subset\mathcal{H}_{\alpha}(C_0,\delta_0,\infty).$$
By the Theorem~\ref{'manifold'} we obtain
$\mathcal{\widetilde{R}}^{i+kN_2}B_{\gamma_1}^{\infty}(\phi_{2r}^*,t_{2r}^*)\subset
B_{\gamma}^{\infty}(\phi_{2r}^*,t_{2r}^*),$ for all $k$ and
$i=1,...,N_2.$ This proves the claim I.
\\\emph{Claim II.} If
$(\phi,t)\in \mathcal{H}_{\alpha}(C,\eta,\infty)$, for some $C >0$
and $\eta
>0$, there exists $N>0$ such that
$$\mathcal{\widetilde{R}}^{N}_{\alpha}(\phi,t)\in
B_{\gamma_1}^{\infty}(\phi_{2r}^*,t_{2r}^*)$$ in the space
$\mathcal{A}_{D_{\delta_0}}\times \mathbb{C}.$ In fact, by
Theorem~\ref{'restate1'} there exists $N_0$ such that for
$\alpha\sim 2r$ and $(\phi,t)\in
\mathcal{H}_{\alpha}(C,\eta,\infty),$
$$\mathcal{\widetilde{R}}^{N_0}_{\alpha}(\phi,t)\in
\mathcal{H}_{\alpha}(C_0,\delta_0,\infty).$$ And from the
Theorem~\ref{'manifold'} for all $\gamma_1>0$ there exists $N_2$
such that
$$\mathcal{\widetilde{R}}^{N_0+N_2}_{\alpha}(\phi,t)\in
B_{\gamma_1}^{\infty}(\phi_{2r}^*,t_{2r}^*).$$ So from the two
claims we have proved
$$\bigcup_{C>0}\bigcup_{\eta >0}\mathcal{H}_{\alpha}(C,\eta,\infty)\subseteq
W^s(\phi_{\alpha}^*,t_{\alpha}^*).$$
\end{bew}

\appendix

\section{Univalent maps}\label{'sec3'}

Here we show some results on the class of univalent functions  in a
domain $U$ containing the interval $[-1,1]$. The proof of the following results can be easily established using basic tools and Koebe's Distortion Theorem (see\cite{Ju} for details).

\begin{theorem}\label{'aproxide'}
For some $K>1$ and $1<\epsilon<K/2$ the following statement holds: if $\phi$
is an univalent map defined in the ball $B(0,K)$ satisfying
$\phi(-1)=-1,$ $\phi(1)=1$ and $\phi([-1,1])\subset [-1,1],$ then
$$|\phi-\mathrm{id}|_{B(0,\epsilon)}<O(\frac{\epsilon}{K}).$$
\end{theorem}

The following lemma was established in~\cite{A-M-W} without a proof. This result is central in the proof of the complex bounds (Theorem 4) for this reason we think that it is convenient to
present a proof of this.

\begin{lemma}[\cite{A-M-W}]\label{'univmap'}
Let $C>0$ be a constant and $E_0 \supset E_1 \supset[-1,1]$ be a
domains strictly contained in the complex plane. There exists a
constant $K>0$ such that the following is satisfied. Let
$\phi:E_0\rightarrow \mathbb{C}$ and $\psi:E_{\psi}\rightarrow
\mathbb{C}$ be univalent maps where $E_{\psi}\subset E_1$  and
furthermore $\phi([-1,1])=[-1,1],$ $\psi([-1,1])=[-1,1]$ and $|\phi
-\mathrm{id}|_{E_0}\leq C.$ There exists a $\rho(\phi)$-stadium
$D_{\rho(\phi)}\subset E_{\psi}$ such that
$$\phi(D_{\rho(\phi)})\subset E_{\psi}.$$
Moreover $$\rho(\phi)\geq e^{-K|\phi
    -\mathrm{id}|_{E_0}}\rho(\psi),$$ where $\rho(\psi)$ is the
distance between the boundary of $E_{\psi}$ to the interval $[-1,1].$
In particular, $\psi\circ\phi$ is defined in $D_{\rho(\phi)}.$
\end{lemma}
\begin{proof} Let $D_{\rho}\subset E_0$ be a $\rho$-stadium. The proof will
be divided in two cases. First suppose that there exists
$\widetilde{C}>1000$ such that $\rho(\psi)>\widetilde{C}.$ So, if $z
\in
\partial D_\rho$ we have the following
$$dist(\phi(z),[-1,1])\leq dist(\phi(z),z)+ dist(z,[-1,1])\leq
|\phi -\mathrm{id}|_{E_0}+\rho$$ then $$\phi(D_\rho)\subseteq
D_{|\phi -\mathrm{id}|_{E_0}+\rho}.$$ Take $\rho(\phi)>0$ such that
$$\rho(\psi)=\rho(\phi)+|\phi -\mathrm{id}|_{E_0}.$$ As $\rho(\psi)>\widetilde{C}$
we have $$\rho(\phi)\geq (1-\frac{1}{\widetilde{C}}|\phi
-\mathrm{id}|_{E_0})\rho(\psi).$$ Then we obtain $K>0$ such that
$$\rho(\phi)\geq e^{-K|\phi -\mathrm{id}|_{E_0}}\rho(\psi).$$
Finally we suppose that $\rho(\psi)\leq 1000.$ Let $x \in [-1,1].$
Using the Generalization Distortion Theorem (~\cite{J.C2}) there
exists $C>0$ such that for all $z,x\in D_{\rho(\psi)}$ $$e^{-C|\phi
-\mathrm{id}|_{E_0}}\leq \frac{|D\phi(z)|}{|D\phi(x)|}\leq e^{C|\phi
-\mathrm{id}|_{E_0}}.$$ On the other hand there exists $x_0 \in
[-1,1]$ such that $D\phi(x_0)=1$ and by the Mean Inequality Theorem on
$E_0$ we have that for all $z\in D_{\rho}$ and $x \in [-1,1]$
$$ |\phi(z)-\phi(x)|\leq e^{C|\phi -\mathrm{id}|_{E_0}}|z-x|.$$
For all $z\in \partial D_{\rho}$ there exists $x \in [-1,1]$ such
that $|z-x|=\rho.$ Then
$$dist(\phi(z),[-1,1])\leq e^{C|\phi -\mathrm{id}|_{E_0}}\rho$$
for all $z\in D_{\rho}.$ Take $\rho(\phi)=e^{-C|(\phi
-\mathrm{id})|_{E_0}}\rho(\psi).$ This finishes the proof.
\end{proof}
\section*{Acknowledgment}
We would like to thank C. Gutierrez, A. Messaoudi, W. de Melo and  E. Vargas  for the very  useful comments and suggestions.

\end{document}